\documentclass[11pt]{article}
\usepackage[margin=1in]{geometry}
 
\usepackage[]{amsmath,amssymb,epsfig}
\usepackage{amsthm}
\usepackage{amsmath}
\usepackage{amssymb}
\usepackage{graphicx}
\usepackage{epstopdf}
\usepackage{comment}
\usepackage{array}
\usepackage{algorithm}
\usepackage{url}

\usepackage{lscape}
\usepackage{algpseudocode}
\usepackage{setspace}
\usepackage{multicol}
\usepackage{multirow}
\usepackage{color}
\usepackage{colortbl}
\usepackage{xcolor}

\usepackage[sort]{cite}

\usepackage{placeins}

\usepackage[%
    font={small,sf},
    labelfont=bf,
    format=hang,    
    format=plain,
    margin=0pt,
    width=0.8\textwidth,
]{caption}
\usepackage[list=true]{subcaption}

\newtheorem{theorem}{Theorem}

\newtheorem{claim}{Claim}

\newtheorem{corollary}{Corollary}

\newtheorem{definition}{Definition}

\newtheorem{lemma}{Lemma}

\newtheorem{remark}{Remark}

\numberwithin{equation}{section}

\DeclareMathOperator{\Tr}{Tr}

\DeclareMathOperator{\diam}{diam}
\DeclareMathOperator{\cl}{cl}
\DeclareMathOperator{\vect}{vec}
\DeclareMathOperator{\polylog}{polylog}

\newcommand{\calA}{\ensuremath{\mathcal{A}}}
\newcommand{\calB}{\ensuremath{\mathcal{B}}}
\newcommand{\calC}{\ensuremath{\mathcal{C}}}

\newcommand{\calS}{\ensuremath{\mathcal{S}}}

\newcommand{\calN}{\ensuremath{\mathcal{N}}}
\newcommand{\calT}{\ensuremath{\mathcal{T}}}
\newcommand{\calU}{\ensuremath{\mathcal{U}}}
\newcommand{\calK}{\ensuremath{\mathcal{K}}}
\newcommand{\calX}{\ensuremath{\mathcal{X}}}

\newcommand{\calKtil}{\ensuremath{\widetilde{\mathcal{K}}}}
\newcommand{\calStil}{\ensuremath{\widetilde{\mathcal{S}}}}

\newcommand{\Mtil}{\ensuremath{\widetilde{M}}}
\newcommand{\rhobar}{\ensuremath{\bar{\varrho}}}

\newcommand{\norm}[1]{\|{#1}\|}
\newcommand{\abs}[1]{\left|{#1}\right|}

\newcommand{\set}[1]{\left\{{#1}\right\}}
\newcommand{\dotprod}[2]{\left\langle#1,#2\right\rangle}
\newcommand{\est}[1]{\widehat{#1}}
\newcommand{\expec}{\ensuremath{\mathbb{E}}}
\newcommand{\matR}{\ensuremath{\mathbb{R}}}

\newcommand{\argmin}[1]{\underset{#1}{\operatorname{argmin}}}

\newcommand{\prob}{\ensuremath{\mathbb{P}}}

\newcommand{\frobsphere}{\mathbb{S}_n}
\newcommand{\frobball}{\calB_n}
\newcommand{\tanconeB}{\calT_{\calK,B}}
\newcommand{\tanconeBtil}{\widetilde{\calT}_{\calK,B}}
\newcommand{\tanconeAstar}{\calT_{\calK,A^*}}
\newcommand{\Xtil}{\widetilde{X}}
\newcommand{\Ubar}{\overline{U}}
\newcommand{\Vbar}{\overline{V}}
\newcommand{\Pbar}{\overline{P}}
\newcommand{\xitil}{\widetilde{\xi}}
\newcommand{\monot}{S^{\uparrow}}

\newcommand{\rev}[1]{\textcolor{black}{#1}} % previous revision after submission to ACHA

\newcommand{\revo}[1]{\textcolor{black}{#1}} %revisions on 30th Sep 2025

\newcommand{\revj}[1]{\textcolor{black}{#1}} %revisions from  11th June 2026 in response to reviews (JoC)

\newcommand{\revjf}[1]{\textcolor{black}{#1}} %revisions from  31st July 2026 in response to reviews (JoC)

\title{Learning linear dynamical systems under convex constraints}

\author{ Hemant Tyagi$^{1}$\\ \texttt{hemant.tyagi@ntu.edu.sg} \and Denis Efimov$^{2}$\\ \texttt{denis.efimov@inria.fr}} 

\date{$^1$Division of Mathematical Sciences, 
      SPMS, 
      NTU Singapore 637371 \\
$^2$\revj{Inria, Univ. Lille, CNRS, UMR 9189 - CRIStAL, F-59000 Lille, France} \newline\newline
\today
} 

\begin{document}
\maketitle

%----------------
% Abstract
%
\begin{abstract}
  We consider the problem of finite-time identification of linear dynamical systems from $T$ samples of a single trajectory. Recent results have predominantly focused on the setup where \revo{either} no structural assumption is made on the system matrix $A^* \in \matR^{n \times n}$, \revo{or specific structural assumptions (e.g. sparsity) are made on $A^*$.}
% 
%, and have consequently analyzed the ordinary least squares (OLS) estimator in detail.
%
We assume prior structural information on $A^*$ is available, which can be captured in the form of a convex set $\calK$ containing $A^*$. For the solution of the ensuing constrained least squares estimator, we derive non-asymptotic error bounds in the Frobenius norm that depend on the local size of $\calK$ at $A^*$. To illustrate the usefulness of these results, we instantiate them for \rev{four} examples, namely when (i) $A^*$ is sparse and $\calK$ is a suitably scaled $\ell_1$ ball; (ii) $\calK$ is a subspace; (iii) $\calK$ consists of matrices each of which is formed by sampling a bivariate convex function on a uniform $n \times n$ grid (convex regression); \rev{(iv) $\calK$ consists of matrices each row of which is formed by uniform sampling (with step size $1/T$) of a univariate Lipschitz function}. In all these situations, we show that $A^*$ can be reliably estimated for values of $T$ much smaller than what is needed for the unconstrained setting.
  
\end{abstract}

% Introduction
%------------------
% Introduction
%------------------
\section{Introduction} \label{sec:intro}
We consider the problem of finite-time identification of a linear dynamical system (LDS) of the form
\begin{equation} \label{eq:lin_dyn_sys_mod}
    x_{t+1} = A^* x_t + \eta_{t+1} \quad \text{ for } t=0,1,\dots,T \quad \text{ and } x_0 = 0,
\end{equation}
where $A^* \in \matR^{n \times n}$ is the unknown system  matrix to be estimated, $x_t \in \matR^n$ is the observed state at time $t$, and $\eta_t \in \matR^n$ is the unobserved (random) process noise.  Such problems arise in many areas such as control theory, reinforcement learning and time-series analysis to name a few. An important line of research in recent years has focused on theoretically analyzing the performance of the ordinary least squares (OLS) estimator, by deriving non-asymptotic error bounds for the estimation of $A^*$ (e.g., \cite{FaraUnstable18,Simchowitz18a, Sarkar19, Jedra20}), holding with high probability provided $T$ is sufficiently large.
The analyses  depend crucially on the spectrum of $A^*$ -- in particular on the spectral radius of $A^*$, namely $\rho(A^*)$. 

The focus of this paper is the strictly stable setting where $\rho(A^*) < 1$. 
Denoting $\Gamma_s(A) = \sum_{k=0}^s A^k (A^k)^{\top}$ for $s \geq 0$ to be the controllability Grammian of the system, and $\lambda_{\min}(\cdot)$ to be the smallest eigenvalue of a symmetric matrix, it was shown recently \cite{Jedra20} that the OLS estimate $\est{A}$ satisfies with probability at least $1-\delta$ 
\begin{equation*}
\norm{\est{A} - A^*}_2 \lesssim \sqrt{\frac{\log(1/\delta) + n}{\lambda_{\min}(\sum_{s=0}^{T-1} \Gamma_s(A^*))}},
\end{equation*}
provided $\lambda_{\min}(\sum_{s=0}^{T-1} \Gamma_s(A^*)) \gtrsim J^2(A^*)(\log(1/\delta) + n)$. Here $\norm{\cdot}_2$ denotes the spectral norm and $(\eta_t)_{t \geq 1}$ are considered to be i.i.d subgaussian vectors -- see Section \ref{prob_setup} for a description of notations. The quantity $J(A^*)$ is defined in \eqref{eq:stab_param} and is finite when $\rho(A^*) < 1$; it is moreover bounded by a constant if $\norm{A^*}_2 < c < 1$ for some constant $c$. Since $\lambda_{\min}(\sum_{s=0}^{T-1} \Gamma_s(A^*)) \geq T$, we can rewrite the above bound as  
\begin{equation} \label{eq:err_bd_jedra}
    \norm{\est{A} - A^*}_2 \lesssim \sqrt{\frac{\log(1/\delta) + n}{T}} \quad \text{ if } \quad T \gtrsim  J^2(A^*)(\log(1/\delta) + n).
\end{equation}
In other words,  a meaningful error bound is ensured provided the length of the trajectory is at least of the order of the dimension $n$. Furthermore, this bound is also optimal in terms of dependence on $\delta, n$ and $T$ \cite{Simchowitz18a}.  

It is natural to consider the scenario where additional structural information is available regarding $A^*$ -- in this case one would expect that incorporating such information in the estimation procedure should lead to an improved performance compared to the vanilla OLS estimator. 
In many cases of interest, $A^*$ actually has an intrinsically low dimensional structure and it is possible to capture this structural information of $A^*$ through a known convex set $\calK$ containing $A^*$. Computationally, the estimate $\est{A}$ is then obtained by the constrained least squares estimator 
\begin{equation} \label{eq:Aest_convex}
 \est{A} \in \argmin{A \in \calK} \sum_{t=0}^{T} \norm{ x_{t+1} - A x_t}_2^2
\end{equation}
which is also a convex program that can typically be solved efficiently in practice. From a statistical perspective, one would expect to be able to improve the error bounds in \eqref{eq:err_bd_jedra} in terms of the dependence on the (extrinsic) dimension $n$. Some  examples of such $\calK$ -- which will also be used later for instantiating our more general results -- are outlined below.
\begin{enumerate}
\item ({\bf Example 1}) $\calK$ is a $d$-dimensional subspace of $\matR^{n \times n}$ for some $d \leq n^2$. This was considered recently in \cite{zheng2021finite} which cited applications of this setup in the time series analysis of spatio-temporal studies and social networks.

\item ({\bf Example 2}) If $A^*$ is $k$-sparse, i.e., has $k$ non-zero entries, then one can choose $\calK$ to be a suitably scaled $\ell_1$ ball such that $A^* \in \calK$. Such assumptions exist in the literature for model \eqref{eq:lin_dyn_sys_mod} as will be discussed in Section \ref{subsec:rel_work}. It is well known in the statistics and signal processing communities that the resulting estimator -- known as the LASSO -- promotes solutions which are sparse (see, e.g.,  \cite{candes09,neghabhan12,candes06}). 

\item ({\bf Example 3}) In this case, we consider $A^*$ to be formed by sampling an unknown convex function $f_0: [0,a]^2 \rightarrow \matR$ on a uniform grid. Then, the set $\calK$ can be written as
\begin{equation*}
  \calK = \set{A \in \matR^{n \times n}: \ A_{ij} = f\left(\frac{\revj{(i-1)a}}{n-1}, \frac{\revj{(j-1)a}}{n-1} \right) \revj{,\; i,j\in[n],} \text{ for some convex } f:[0,a]^2 \rightarrow \matR}.   
\end{equation*}
This problem is similar to the problem of convex-regression that has been studied extensively in the statistics community in the i.i.d setting, see,  e.g., \cite{hildreth54, Kur2020ConvexRI,Deme17} and references therein. From a computational perspective, it is useful to know that $\calK$ can be characterized as a convex cone \cite[Lemma 2.2]{Seijo11} which thus leads to a convex quadratic program that can be solved efficiently through standard solvers in practice; there also exist efficient specialized solvers for convex regression (e.g., \cite{chen2021multivariate}).

\item \rev{({\bf Example 4}) In this case, each row of $A^*$ is formed by uniform sampling (with step size $1/T$) of a Lipschitz function $f_i: [0,1] \rightarrow \matR$ (with Lipschitz constant $L_i$) for \revj{$i \in [n]$}. Then the set $\calK$ can be written as
\begin{equation*}
    \calK =  \set{A \in \matR^{n \times n}:  \abs{A_{i,j} - A_{i,j+1}} \leq \frac{L_i}{T}, \ i \in [n], \ j \in [n-1]}.
\end{equation*}
This example is motivated by a similar example proposed in \cite{neykov19b} for structured linear regression from i.i.d observations.
}
\end{enumerate}
In Examples $1,2$, the intrinsic dimension of $A^*$ is essentially captured by the quantities $d$ or $k$, and we expect that the  error bounds in \eqref{eq:err_bd_jedra} should improve in terms of exhibiting a milder dependence on $n$. In particular, when $d, k \ll n^2$, we expect the estimation error for $A^*$ to be small for moderately large values of $T$. \rev{While this is less obvious for Examples $3$ and $4$, we will see later that similar conclusions can be drawn; for Example $3$ this will be shown when $f_0$ is piecewise affine with a `simple structure'}.

\subsection{Our contributions} 
For the setting where $\rho(A^*) < 1$ and $A^* \in \calK$, we derive non-asymptotic bounds on the estimation error in the Frobenius norm $\norm{\est{A} - A^*}_F$ for the estimator \eqref{eq:Aest_convex}, holding with high probability; see Theorems \ref{thm:main_err_tangent_cone} and \ref{thm:no_tancone_Struc} for the full statements, \revj{which also constitute our main results}. Our bounds depend on the `local size' of the set $\calK$ at $A^*$, captured via Talagrand's $\gamma_1, \gamma_2$ functionals \cite{talagrand2014upper} (see Definitions \ref{def:gamma_fun} and \ref{def:tancone} in Section \ref{prob_setup}). Upon instantiating our bounds for the aforementioned choices of $\calK$, we obtain the following corollaries.
\begin{enumerate}
    \item ({\bf Example 1}) In this case, Theorem \ref{thm:main_err_tangent_cone} states (see Corollary \ref{cor:subspace}) that with probability at least $1-\delta$, 
    \begin{equation} \label{eq:intro_subsp_bd_corr}
      \norm{\est{A} - A^*}_F \lesssim   J(A^*)\left(\frac{\log(1/\delta) + \sqrt{d}}{\sqrt{T}} \right) \quad \text{if} \quad T \gtrsim J^4(A^*) \max\set{d, \log^2(1/\delta)} .
    \end{equation}
   Suppose for simplicity that $\norm{A^*}_2 < 1$ so that $J(A^*)$ is a constant. If $d = n^2$, we obtain the rate $\frac{n}{\sqrt{T}}$ which matches that obtained from \eqref{eq:err_bd_jedra}  using the standard inequality $\norm{\est{A} - A^*}_F \leq \sqrt{n} \norm{\est{A} - A^*}_2$. Moreover, we would then also need $T \gtrsim n^2$ in \eqref{eq:err_bd_jedra} in order to drive $\norm{\est{A} - A^*}_F$ below a specified threshold. For general $d$, however, the sample complexity of estimating $A^*$ is seen to be of order $d$ which is relevant when $d \ll n^2$.

    \item ({\bf Example 2})  In this case we obtain Corollary \ref{cor:sparse_example} of Theorem \ref{thm:main_err_tangent_cone} which is best interpreted for specific regimes of the sparsity level $k$. For instance, if $k$ is of the order $n$, we show that 
      \begin{equation*}
       \norm{\est{A} - A^*}_F \lesssim J(A^*) \left(\frac{\log(1/\delta) + \sqrt{n \log n}}{\sqrt{T}} + \frac{(n \log n)^{3/2}}{T} \right) 
   \end{equation*}
   \rev{with probability at least $1-\delta$}, if $T \gtrsim J^4(A^*) \max\set{n \log n, \log^2(1/\delta)}$. Assuming $J(A^*)$ is constant, note that $T \gtrsim (n \log n)^{3/2}$ suffices to drive $\norm{\est{A} - A^*}_F$ below a specified threshold, however, this is still much milder than what we need in general.

   \item ({\bf Example 3}) Suppose $f_0$ is piecewise affine convex with $K$ pieces, and where the $i$th piece lies on a convex subdomain $\Omega_i$ with $\Omega_1,\dots,\Omega_K$ a partition of the domain $[0,a]^2$. Suppose each $\Omega_i$ is formed by the intersection of at most $s$ pairs of parallel halfspaces\footnote{We use the terminology of \cite{Kur2020ConvexRI} here. A pair of parallel halfspaces is the set $\set{v \in \matR^2: a \leq w^\top v \leq b}$ for scalars $a < b$ and a unit vector $w \in \matR^2$.}. Then Corollary \ref{cor:convex_reg_biv} of Theorem \ref{thm:no_tancone_Struc} implies that if $n$ is large enough and $T \gtrsim J^4(A^*) \max \set{c^s K \log^s n, \log^2\left(\frac{1}{\delta}\right)}$  then with probability at least $1-\delta$ 
     \begin{equation*}
      \norm{\est{A} - A^*}_F \lesssim J(A^*) \left(\frac{K c^s \log^{s+1} n}{T} + \frac{\log(1/\delta)}{\sqrt{T}}\right)   
     \end{equation*}
  for some constant $c > 0$. So if $J(A^*), K$ and $s$ are constants then it suffices that $T \gtrsim \polylog n$ to drive the error below a specified threshold.  

  \item \rev{({\bf Example 4}) In this case, with $L$ such that $L_i \leq L$ for all $i \in [n]$,  Corollary \ref{cor:lipschitz_reg_examp} of Theorem \ref{thm:main_err_tangent_cone} roughly states that $\norm{\est{A} - A^*}_F \leq \epsilon$ is ensured with probability at least $1-\delta$, if
  \begin{equation*}
      T \gtrsim \revj{J^4(A^*) \max \set{\frac{\log^2(1/\delta)}{\epsilon^2},   \frac{\revjf{(n \log n)^{3/2}}}{\epsilon^{3/2}}, 
  L^{1/3} \frac{n^{5/3} \log n }{\epsilon}}.}
  \end{equation*}
  If $J(A^*)$ and $\epsilon$ are constants, then $T = \revj{\Omega(n^{5/3} \log n)}$ samples are sufficient for accurate estimation of $A^*$.
  }
\end{enumerate}
\revj{To put the above examples in perspective, we remark that the bound in Corollary \ref{cor:subspace} matches with the one obtained in \cite{zheng2021finite} up to a slightly worse dependence on $\log(1/\delta)$. This is discussed in more detail in Section \ref{subsec:rel_work}.} 

\revj{The error rate in Corollary \ref{cor:sparse_example} is unfortunately sub-optimal which is essentially due to the appearance of the $\gamma_1$ functional in Theorem \ref{thm:main_err_tangent_cone}; this issue is discussed later in Remark \ref{rem:subopt_sparse_example} as well. There exist results for this particular setting in the literature  which achieve the optimal error rate. This is discussed in Remark \ref{rem:comp_sparseA_lit} for the work \cite{Basu15}.}

\revj{The statements of Corollaries \ref{cor:convex_reg_biv} and \ref{cor:lipschitz_reg_examp} are new, to the best of our knowledge, in the context of identification of structured linear dynamical systems.}

%----------------
% Related work
%----------------
\subsection{Related work} \label{subsec:rel_work}
\paragraph{Learning unstructured LDS.} A recent line of work has focused on deriving non-asymptotic error bounds for learning linear systems of the form \eqref{eq:lin_dyn_sys_mod}, without any explicit structural assumption on $A^*$. The majority of these works analyze the OLS under different assumptions on $\rho(A^*)$, namely: strict stability ($\rho(A^*) < 1$) \cite{Jedra20,FaraUnstable18}; marginal stability ($\rho(A^*) \leq 1$) \cite{Simchowitz18a, Sarkar19}; purely explosive systems ($\rho(A^*) > 1$) \cite{FaraUnstable18,Sarkar19}. While $\est{A}$ is known in closed form, the main challenge in the analysis comes from handling the interaction between the matrix of covariates $x_t$, and that of noise terms $\eta_t$ due to their dependencies. Common techniques used in the analysis involve concentration results for self normalized processes \cite{selfnormbook, abbasi11}, and Mendelson's ``small-ball'' method \cite{pmlr-v35-mendelson14}, the latter of which was extended to dependent data in \cite{Simchowitz18a} leading to sharper error bounds. When $\rho(A^*) \leq 1$, the authors in \cite{Simchowitz18a} interpret the quantity $\lambda_{\min}(\Gamma_{T-1})$ as a measure of the signal-noise-ratio \cite{Simchowitz18a} -- larger values  lead to improved error bounds. As mentioned earlier, the results of \cite{Jedra20} depend on a similar quantity, namely 
$\lambda_{\min}(\sum_{s=0}^{T-1} \Gamma_{s})$, which plays a key role in their error bounds. These terms do not appear explicitly within our \revj{results, but this is likely an artefact of our analysis.} 
%and it is unclear (albeit interesting) how this can be done. 
The main tools that we employ involve concentration results for the suprema of second-order subgaussian chaos processes indexed by a set of matrices \cite{krahmer14, dirksen15}; see Section \ref{subsec:tools} for details. 

\paragraph{Learning structured LDS.} Existing works for learning structured LDS are typically focused on specific structures such as sparsity, group sparsity, or low-rank models.
\begin{itemize}
\item In \cite{Fattahi2019}, a more general version of \eqref{eq:lin_dyn_sys_mod} was considered where $x_{t+1} = A^* x_t + B^* u_t + \eta_t$, with $B^* \in \matR^{n \times m}$ and $u_t \in \matR^{m}$ denoting the inputs. Assuming the unknown $A^*, B^*$ to be $k$-sparse, and $u_t = K_0 x_t + v_t$ where $v_t$ is random with a user specified distribution ($K_0$ is a feedback controller), a LASSO type estimator was analyzed. Assuming $x_0$ rests at its stationary distribution, uniform asymtotic stability of the closed-loop system, and certain technical assumptions involving $A^*, B^*$ and $K_0$, entry-wise error bounds were obtained for the estimation of $A^*, B^*$. It was shown that these bounds can sometimes be obtained with $T$ of the order $k^2 \log (n + m)$. If $k$ is of order $n$, this means that $T \gtrsim n^2 \log (n+m)$ samples are needed for recovering the \emph{support} of $A^*, B^*$. This is larger than our sample complexity bound for controlling the Frobenius norm error. 

\item In \cite{Pereira10}, the model \eqref{eq:lin_dyn_sys_mod} was considered with $A^*$ assumed to be $k$-sparse and strictly stable. Under certain assumptions on the problem parameters, it was shown for a LASSO-type estimator that the support of $A^*$ is recovered exactly provided $T \gtrsim \revj{k^3}\log n$. \revj{If $k$ is of order $n$, then this translates to $T \gtrsim n^3 \log n$ which is worse than our obtained scaling $(n \log n)^{3/2}$ for controlling the Frobenius norm error.} 

\item The model \eqref{eq:lin_dyn_sys_mod} is a vector autoregressive (VAR) model of order $1$. In \cite{Melnyk16}, a more general order $d$-VAR model\footnote{Note that an order $d$-VAR can be written as an order-$1$ VAR in a higher dimensional space.} was considered involving matrices $A^*_1,\dots,A^*_d$. They analyzed a penalized least squares method where the penalty is imposed by a norm-based regularizer $R(\cdot)$; examples of $R$ including: $\ell_1$ norm, group LASSO, sparse group LASSO and OWL (ordered weighted $\ell_1$ norm). The recent work \cite{Wang23} considered an order \revj{$d$-VAR} as well, but allowed for heavy-tail $\eta_t$'s and model-misspecification. They proposed an estimator based on the Yule-Walker equation that deploys a regularizer (norm-based) $R$, this leads to near minimax optimal results for estimating the model parameters under different structural settings: 
sparsity, low-rankness, and linear constraints. Other existing results primarily make sparsity-type assumptions on the parameters \cite{Loh12,Basu15,Kock2015,song2011large}. The recent work \cite{Basu19} considered the model \eqref{eq:lin_dyn_sys_mod} and assumed that $A^*$ has a low-rank plus sparse decomposition, they analyzed a penalized least squares approach for recovering the low-rank and sparse components of $A^*$. \revj{If $\calK$ is the $\ell_1$ ball and $A^*$ is sparse (i.e., Example $2$), then the instantiation of our general bound in Theorem \ref{thm:main_err_tangent_cone}  leads to sub-optimal rates (cf. Corollary \ref{cor:sparse_example}) in terms of the dependence on dimension $n$. See Remarks \ref{rem:subopt_sparse_example} and \ref{rem:comp_sparseA_lit} for further discussion.}

\item The results in \cite{zheng2021finite} are applicable to model \eqref{eq:lin_dyn_sys_mod}, with additional linear information about $A^*$ assumed to be available. This can be reformulated as saying that for a known $d$-dimensional basis $\set{V_i}_{i=1}^d \subset \matR^{n \times n}$ and a known offset $\bar{V} \in \matR^{n \times n}$, we have $A^* - \bar{V} \in \text{span}\set{V_i}_{i=1}^d$. This is identical to Example $1$. If $\rho(A^*) < 1$ and $\norm{A^*}_2 \leq C$ for some constant $C > 0$, they show \revj{\cite[Theorem 3(ii)]{zheng2021finite}} that $\norm{A^* - \est{A}}_F \lesssim \sqrt{\frac{d\log(d/\delta)}{T}}$ provided the smallest singular value of $A^*$ is  smaller than $1$.  This is \revj{slightly worse than} our bound in \eqref{eq:intro_subsp_bd_corr}, \revj{however they also show \cite[Theorem 3(iii)]{zheng2021finite} that if the condition on $\norm{A^*}_2$ is replaced by the stronger condition that for some $\kappa \in (0,1)$, $\norm{(A^*)^t}_2 \leq C \kappa^t$ holds\footnote{\revj{\revjf{If this holds for all integers $t \geq 1$}, note that $J(A^*)$ as defined in \eqref{eq:stab_param} is bounded by \revjf{$1 + \frac{C \kappa}{1-\kappa}$}.}} for all integers $1 \leq t \leq T$, then\footnote{\revj{\cite[Theorem 3(iii)]{zheng2021finite} requires $(\eta_t)_{t \geq 1}$ to be normal random variables.}} $\norm{A^* - \est{A}}_F \lesssim \sqrt{\frac{d + \log(1/\delta)}{T}}$. This bound has a better dependence on $\log(1/\delta)$ than \eqref{eq:intro_subsp_bd_corr}. Moreover,} they also cover the \revj{more challenging} setting $\rho(A^*) \leq 1$ (\revj{marginal stability}) where the analysis uses the small ball method \cite{pmlr-v35-mendelson14}, \revj{and makes explicit use of the closed form expression of $\est{A}$ in the analysis}. \revj{Our analysis does not apply to the marginal stability setting, and this seems particularly difficult to achieve for general $\calK$. Finally, we note that \cite{zheng2021finite} also provides minimax lower bounds on the estimation error when $(\eta_t)_{t \geq 1}$ are normally distributed, for different stability regimes.}
\end{itemize}
% = \set{A : R(A) \leq R(A^*)}$
\rev{The work \cite{PGDstruct21} studies the model \eqref{eq:lin_dyn_sys_mod} and analyses the estimator \eqref{eq:Aest_convex} when $\calK$ is defined via a convex regularizer $R(\cdot)$. The projected gradient descent (PGD) method for solving \eqref{eq:Aest_convex} is analyzed under suitable stability assumptions on $A^*$. It is shown that the iterates of PGD converge linearly (in the Frobenius norm) to the statistical error, as measured by the Gaussian width of the (local) tangent cone of $\calK$ at $A^*$, for $T$ suitably large. We discuss this in more detail later in Remark \ref{rem:rel_work_comp} after detailing our results.}

\paragraph{Learning structured signals from random linear measurements.} Consider the relatively easier setting of linear regression\footnote{\revj{Note that \eqref{eq:lin_dyn_sys_mod} can be rewritten (using Kronecker products) in the form of \eqref{eq:struc_signal_rec_ind}, where $\beta^*$ will correspond to the vectorized version of $A^*$. The important difference is that as opposed to the setting of (1.5) -- where $X$ is independent of $\eta$ and is typically assumed to have independent rows -- this reformulation of (1.1) leads to a linear model where $X$ is dependent on $\eta$ and also has dependent rows.}} with independent covariates and noise, i.e., 
\begin{equation} \label{eq:struc_signal_rec_ind}
y = X \beta^* + \eta
\end{equation}
where $X \in \matR^{m \times n}$ is the matrix of covariates, $\beta^* \in \matR^n$ is the unknown signal,  $\eta$ is noise, and $X$ is \emph{independent} of $\eta$. Suppose for simplicity that the entries of $X$ and $\eta$ are i.i.d centered Gaussian's. The problem of recovering $\beta^*$ -- assuming it belongs to a convex set $\calK \subseteq \matR^n$ -- has received significant interest over the past decade from the statistical and signal processing communities. It is now known that the efficient recovery of $\beta^*$ is possible via convex programs (e.g., via solving least squares with constraint $\calK$, or by penalizing the objective) with the sample complexity $m$ depending on the Gaussian width of the `local size' of $\calK$ at $\beta^*$; see for e.g., \cite{neykov19b, rudelson08, chandra12, Tropp2015, planlasso16} and also \cite{lotz14} who introduced a related notion of `statistical dimension'. For some sets $\calK$ (such as the $\ell_1$ ball), sharp estimates for the Gaussian width are available through tools such as Gordon's escape through the mesh theorem \cite{Gordon88}, that leads to tight sample complexity bounds. While our proof technique is similar in spirit to these papers (in particular \cite{neykov19b,planlasso16}), the model in \eqref{eq:lin_dyn_sys_mod} leads to additional technical difficulties. For instance, we cannot use Gordon's theorem anymore and require other concentration tools for the underlying second order subgaussian chaos. To our knowledge, existing works for finite-time identification of \eqref{eq:lin_dyn_sys_mod} \revj{typically} do not provide bounds for general closed convex sets $\calK$, \revj{see also Remark \ref{rem:rel_work_comp}}. Our main goal is to fill this gap (to an extent) by drawing ideas from the above literature.  

\paragraph{Shape constrained regression.} There is a rich line of work for the Gaussian sequence model  
\begin{equation*}
  y_i = \beta_i^* + \eta_i; \quad i=1,\dots,n,
\end{equation*}
where $\beta^*_i = f(x_i)$ (with $x_i$'s the design points) for some unknown function $f \in \mathcal{F}$, with the function class $\mathcal{F}$ known. This is known as `shape constrained' nonparametric regression since the shape of $f$ (depending on $\mathcal{F}$) imparts structure to $\beta^*$ of the form $\beta^* \in \mathcal{K} \subset \mathbb{R}^n$ where 
$$\calK := \set{(f(x_1),\dots,f(x_n)) \in \matR^n: \ f \in \mathcal{F}}.$$
When $\eta_i$'s are i.i.d centered Gaussians, then the maximum likelihood estimator (MLE) of $\beta^*$ would be the orthogonal projection of $y$ on $\calK$. 
There exist many interesting examples of $\mathcal{F}$ which impart an intrinsically low-dimensional structure to $\mathcal{K}$, thus leading to better error bounds (in terms of $n$) for the MLE as compared to the unconstrained case. Some examples include: convex functions (e.g., \cite{Chatter15,Kur2020ConvexRI}), isotonic functions (e.g., \cite{Chatter15,Bellec18}) and Lipschitz functions \cite{neykov19b}; see \cite{Guntu18Survey} for a nice overview of this area and also the seminal work of Chatterjee \cite{Chatter14Persp}. In this respect, our work can be interpreted as shape constrained regression for the vector autoregressive model in \eqref{eq:lin_dyn_sys_mod}.

% Problem setup and results
%-------------------
% Problem setup
%-------------------
\section{Problem setup and results} \label{prob_setup}
%Define the problem setup: notation, model, goal etc. State main result(s).
\paragraph{Notation.} For any vector $x \in \matR^n$, $\norm{x}_p$ denotes the usual $\ell_p$ norm of $x$. For $X \in \matR^{n \times m}$, we denote $\norm{X}_2$, $\norm{X}_F$ to be respectively the spectral and Frobenius norms of $X$, while $\dotprod{X}{Y} = \Tr(X^\top Y)$ denotes the inner product between $X$ and $Y$. Also, $\vect(X) \in \matR^{nm}$ is formed by stacking the columns of $X$ and $\norm{X}_{1,1} = \norm{\vect(X)}_1$ denotes the entry-wise $\ell_1$ norm of $X$. The symbol $\otimes$ denotes the Kronecker product between matrices. Sets will be usually denoted by calligraphic letters, and their cardinalities by $\abs{\cdot}$. For any set $\calS \subset \matR^{n \times m}$ and $X \in \matR^{n \times m}$, we will denote 
$$\calS - X = \set{A-X: A \in \calS}.$$
\revj{Moreover, for any integer $n \geq 1$, we denote the set $[n] = \set{1,2,\dots,n}$.} For $a,b > 0$, we say $a \lesssim b$ if there exists a constant $C > 0$ such that $a \leq C b$. If $a \lesssim b$ and $a \gtrsim b$, then we write $a \asymp b$. For $n \times n$ matrices, we denote the unit Frobenius sphere by $\frobsphere$, the unit Frobenius ball by $\frobball$, and the identity matrix by $I_n$. The values of symbols used for denoting constants (e.g., $c, C, c_1$ etc.) may change from line to line. Finally, recall that the subgaussian norm of a random variable $X$ is given by 
$\norm{X}_{\psi_2}:= \sup_{p \geq 1} p^{-1/2} (\expec \abs{X}^p)^{1/p}$, see for e.g. \cite{vershynin_2012}. We say $X$ is $L$-subgaussian if $\norm{X}_{\psi_2} \leq L$.

%-----------
% Setup
%-----------
\subsection{Setup}
Consider the autonomous linear dynamical system in \eqref{eq:lin_dyn_sys_mod} where $(\eta_t)_{t \geq 1}$ are assumed to be zero-mean, independent and identically distributed (i.i.d) random variables for $t=0,\dots,T$. Specifically, $\eta_t$ is assumed to have i.i.d $L$-subgaussian entries (for some constant $L$), each of unit variance. Given $(x_t)_{t=0}^{T+1}$, our goal is to estimate $A^*$ under the constraint that $A^* \in \calK$ for a closed convex set $\calK \subseteq \matR^{n \times n}$. We focus on the estimator  \eqref{eq:Aest_convex} which is a convex program that can usually be efficiently solved by interior-point methods, and in many cases by more efficient methods (e.g., projected gradient descent) specialized to the structure of $\calK$. Also, if $\eta_t$ were i.i.d Gaussian's, then \eqref{eq:Aest_convex} would simply correspond to the maximum likelihood estimator (MLE) of $A^*$.

In our analysis, we will assume $A^*$ is strictly stable, i.e., its spectral radius $\rho(A^*) < 1$. The quantity $J(A^*)$ defined as 
\begin{equation} \label{eq:stab_param}
  J(A^*) := \sum_{i=0}^{\infty} \norm{(A^*)^i}_2  
\end{equation}
was introduced in \cite{Jedra20} for the analysis of the OLS estimator for strictly stable linear dynamical systems, and will also appear in our results. It is not difficult to verify that $J(A^*)$ is \revj{finite} if $\rho(A^*) < 1$, although it could grow with $n$. If $\norm{A^*}_2 < 1$ then $J(A^*) \leq \revo{\frac{1}{1-\norm{A^*}_2}}$. \revj{As noted in Section \ref{subsec:rel_work}, if for some $\kappa \in (0,1)$ and $C > 0$, $\norm{(A^*)^t}_2 \leq C \kappa^t$ holds for all integers \revjf{$t \geq 1$}, then $J(A^*) \leq \revjf{1 + \frac{C \kappa}{1-\kappa}}$.}
%
%
%\hemant{(Actually, this needs to be checked in more detail -- how to bound $J(A^*)$ in terms of $\rho(A^*)$ if $A^*$ is stable?)} \denis{Using Jordan normal form theorem: there is a non-singular matrix $P\in\mathbb{R}^{n\times n}$ such that $A^*=P^{-1}J^* P$, where $J^*$ is a block-diagonal matrix having on the main diagonal real Jordan blocks. Hence, $(A^*)^i=P^{-1}(J^*)^i P$ and $\norm{(A^*)^i}_2\leq\norm{P^{-1}}_2\norm{(J^*)^i}_2 \norm{P}_2\leq\norm{P^{-1}}_2\norm{P}_2 \rho(A^*)^i$, which allows us to get a converging upper bound for $J(A^*)$.}
%\hemant{I think there are two problems with the above argument. The first is that $\norm{P^{-1}}_2$, $\norm{P}_2$ can grow with $n$, even as fast as polynomial with $n$, which makes the error bounds appearing later quite bad; note that the dependence on $n$ is crucial in what we say later. Also, I am not sure if $\norm{J^*}_2 \leq \rho(A^*)$ is true? See for e.g., this link: \url{https://math.stackexchange.com/questions/853143/2-norm-of-a-canonical-jordan-form-and-spectral-radius}} 

Before stating our results, we need to present some definitions which will be used later on. 

%
% Preliminaries
%
\subsection{Preliminaries} \label{subsec:prelim}
We begin by recalling Talagrand's $\gamma_{\alpha}$ functionals \cite{talagrand2014upper} which can be interpreted as a measure of the complexity of a (not necessarily convex) set.
\begin{definition}[\cite{talagrand2014upper}] \label{def:gamma_fun}
Let $(\calS,d)$ be a metric space. We say that a sequence of subsets of $\calS$, namely $(\calS_r)_{r \geq 0}$ is an admissible sequence if $\abs{\calS_0} = 1$ and $\abs{\calS_r} \leq 2^{2^{r}}$ for every $r \geq 1$. Then for any $0 < \alpha < \infty$, the $\gamma_{\alpha}$ functional of $(\calS, d)$ is defined as 
\begin{equation*}
    \gamma_{\alpha}(\calS, d) := \inf \sup_{s \in \calS} \sum_{r=0}^{\infty} 2^{r/\alpha} d(s, \calS_r)
\end{equation*}
with the infimum being taken over all admissible sequences of $\calS$. %\denis{Explain the notation $|\mathcal{S}_r|$?} \hemant{Added in notations paragraph.}
\end{definition}
The following properties of the $\gamma_{\alpha}$ functional are useful to note.
\begin{enumerate}
\item It can be verified that for any two metrics $d_1, d_2$ such that $d_1 \leq a d_2$ for some $a > 0$,  it holds
%\footnote{See for instance \cite[Exercise 2.2.20]{talagrand2014upper} for $\alpha = 2$, the arguments however extend to $\alpha \geq 1$.} 
that $\gamma_{\alpha}(\calS, d_1) \leq a \gamma_{\alpha}(\calS, d_2)$.

\item For $\calS' \subset \calS$, we have that $\gamma_{\alpha}(\calS', d) \leq C_{\alpha} \gamma_{\alpha}(\calS, d)$ for $C_{\alpha} > 0$ depending only on $\alpha$.

\item \revo{If $f: (\calT, d_1) \rightarrow (\calU, d_2)$ is onto and for some constant $M$ satisfies
\begin{equation*}
    d_2(f(x), f(y)) \leq M d_1(x,y) \quad \text{for all } x,y \in \calT,
\end{equation*}
then $\gamma_{\alpha}(\calU, d_2) \leq C_{\alpha} M \gamma_{\alpha}(\calT, d_1)$ for $C_{\alpha} > 0$ depending only on $\alpha$.}
\end{enumerate}
\revo{Properties $2$ and $3$ above are stated in \cite[Theorem 1.3.6]{tal05} (with $C_{\alpha} = 1$) for an alternative definition of the $\gamma_{\alpha}$ functional \cite[Definition 2.2.19]{talagrand2014upper} which is equivalent to Definition \ref{def:gamma_fun} up to a constant depending only on $\alpha$; see \cite[Section 2.3]{talagrand2014upper}.}

The $\gamma_{\alpha}$ functionals can be upper bounded in terms of the covering numbers of the set $\calS$. For any $\epsilon > 0$, denote $\calN(\calS, d, \epsilon)$ to be the minimum number of balls of radius $\epsilon$ (with centers in $\calS$) which are needed to cover $\calS$. Then, one can show\footnote{This can be deduced using  \cite[Corollary $2.3.2$]{talagrand2014upper}, and by replicating the arguments after the proof of \cite[Lemma 2.2.11]{talagrand2014upper} to general $\alpha \geq 1$.} 
that
\begin{equation} \label{eq:gamma_bound}
    \gamma_{\alpha}(\calS, d) \leq c_{\alpha} \int_{0}^{\diam(\calS)} \log^{1/\alpha} \calN(\calS, d, \epsilon) d\epsilon
\end{equation}
where $c_{\alpha} > 0$ depends only on $\alpha$, and $\diam(\calS)$ is the diameter of $\calS$.  
For $\alpha = 2$, the right-hand side (RHS) of \eqref{eq:gamma_bound} is the well-known Dudley entropy integral \cite{Dudley67}. In fact, by Talagrand's majorizing measure theorem \revo{\cite[Theorem 2.4.1]{talagrand2014upper}}, $\gamma_2(\calS,d)$ characterizes the expected suprema of centered Gaussian processes $(X_s)_{s \in \calS}$ as
\begin{equation} \label{eq:talag_maj_meas_thm}
  c \gamma_2(\calS,d) \leq \expec \sup_{s \in \calS}  X_s \leq C \gamma_2(\calS,d)
\end{equation}
for some universal constants $c, C > 0$, with the canonical distance $d(s,s') := (\expec[X_s - X_{s'}]^2)^{1/2}$. For example, if $\calS \subset \matR^{n \times m}$, and $X_s = \dotprod{G}{s}$ for a $n \times m$ matrix $G$ with iid standard Gaussian entries, we have $d(s,s') = \revj{\norm{s - s'}_F}$. Then \eqref{eq:talag_maj_meas_thm} implies $\expec \sup_{s \in \calS} \dotprod{G}{s} \asymp \gamma_2(\calS, \norm{\cdot}_F)$ where $\expec \sup_{s \in \calS} \dotprod{G}{s}$ is known as the \emph{Gaussian width} of the set $\calS$, denoted as $w(\calS)$. %\denis{You use the symbol $\asymp$ just to say that two quantities are of the same order, but defined it for sequences? Explain what does it mean also for not sequences?} \hemant{Yes you are right, I think its better to just define it for two positive numbers $a,b$. Corrected this in notations paragraph. Does it look ok now?}
\paragraph{Tangent cone.}
One of our sample complexity bounds for estimating $A^*$ will depend on the local size of the \emph{tangent cone} of the set $\calK$ at $A^*$. 
\begin{definition}[Tangent cone] \label{def:tancone}
For a convex set $\calK$  and $A \in \calK$, the tangent cone at $A$ is defined as 
\begin{equation*}
    \calT_{\calK,A} := \cl\set{t(B - A): t \geq 0, \ B \in \calK}
\end{equation*}
where $\cl(\cdot)$ denotes the closure of a set.
\end{definition}
As we will see shortly, our first main result, namely Theorem \ref{thm:main_err_tangent_cone} below, will involve the gamma functionals $\gamma_1(\calT_{\calK,A^*} \cap \frobsphere, \norm{\cdot})$ and $\gamma_2(\calT_{\calK,A^*} \cap \frobsphere, \norm{\cdot})$ with $\norm{\cdot}$ corresponding to either the spectral or Frobenius norm. Small values of these terms will translate to weaker requirements on the sample size $T$ for accurately estimating $A^*$. By virtue of the earlier discussion, note that $\gamma_2(\calT_{\calK,A^*} \cap \frobsphere, \norm{\cdot}_F) \asymp w(\calT_{\calK,A^*} \cap \frobsphere)$.

%----------------------
% Main results 
%----------------------
\subsection{Main results}
Our first main result is the following theorem which bounds the estimation error $\norm{\est{A} - A^*}_F$. The proofs of our main results are detailed in Section \ref{sec:proof}.
\begin{theorem}[Tangent cone structure] \label{thm:main_err_tangent_cone}
There exist constants $C_1,C_2, C_3,C_4 > 0$ depending only on $L$ such that for any $\delta \in (0,1)$ and $B\in\calK$, if
$$T \geq C_1 J^4(A^*) \max\set{\gamma_2^2(\tanconeB \cap \frobsphere, \norm{\cdot}_2), \log^2(C_2/\delta)},$$
then it holds with probability at least $1-\delta$ that 
\begin{align*}
    \norm{\est{A} - A^*}_F 
    &\leq C_3 J(A^*)\left(\frac{\log(C_2/\delta) + \gamma_2(\tanconeB \cap \frobsphere, \norm{\cdot}_F)}{\sqrt{T}} + \frac{\gamma_1(\tanconeB \cap \frobsphere, \norm{\cdot}_2)}{T} \right) \\
    &+ C_4 J^2(A^*) \norm{A^* - B}_F .
\end{align*}
\end{theorem}
In the formulation of this theorem a generic matrix $B \in \calK$ is introduced, and a natural choice for it is $B=A^*$, which minimizes the last term in RHS above. However, the shape of $\tanconeB$, which becomes important in evaluation of $\gamma_{\alpha}(\tanconeB \cap \frobsphere, \norm{\cdot}_2)$ for $\alpha=1,2$, may be more suitable for calculation if $B\ne A^*$ (but sufficiently `close' to $A^*$). We will illustrate this on \rev{Examples $1,2$ and $4$} in the next section. \rev{In particular, we will take $B = A^*$ in Examples $1$ and $2$. For Example $4$, we will require a particular construction of $B \neq A^*$ such that $\gamma_{\alpha}(\tanconeB \cap \frobsphere, \norm{\cdot}_F)$ is suitably ``small''.}

Our second main result is useful in situations where $\calK$ does not have a small tangent cone structure at $A^*$, or at points in $\calK$ sufficiently close to $A^*$.
%
% without tangent cone structure
%
\begin{theorem}[without tangent cone structure]\label{thm:no_tancone_Struc} 
There exist constants $C_1,C_2, C_3, C_4 > 0$ depending only on $L$ such that the following is true. For any $\delta \in (0,1)$, $x > 0$ and $B \in \calK$, suppose that 
\begin{equation*}
T \geq C_1 J^4(A^*) \max \set{\frac{\gamma_2^2((\calK-B) \cap x\frobball,\norm{\cdot}_2)}{x^2}, \log^2(C_2/\delta)}. 
\end{equation*}
Then with probability at least $1-\delta$, the estimate $\est{A}$ satisfies
\begin{align*}
\norm{\est{A} - A^*}_F &\leq C_3 J(A^*) \left[\frac{\log(C_2/\delta)}{\sqrt{T}} + \frac{\gamma_1((\calK-B) \cap x\frobball, \norm{\cdot}_2)}{T x} + \frac{\gamma_2((\calK-B) \cap x\frobball, \norm{\cdot}_F)}{\sqrt{T} x} \right] \\
&+ C_4 J^2(A^*) \norm{A^*-B}_F + x.
\end{align*}
\end{theorem}
Suppose we take $B = A^*$ and say $A^*$ lies in the interior of $\calK$. Then $\calT_{\calK,A^*} = \matR^{n \times n}$ which renders Theorem \ref{thm:main_err_tangent_cone} to be unhelpful. In such a situation, Theorem \ref{thm:no_tancone_Struc} can be meaningful in terms of capturing the local size of $\calK$ at $A^*$ at a suitably small scale $x = o(1)$ (as $T \rightarrow \infty$). The optimal choice of $x$ will depend on how the terms involving the $\gamma_1$  and $\gamma_2$ functionals scale with $x$. We will apply Theorem \ref{thm:no_tancone_Struc} on Example $3$ in the next section. Interestingly, we will see that both the $\gamma_1$ and $\gamma_2$ terms are bounded by quantities depending \emph{linearly} on $x$. In this case, the limit $x = 0$ is the optimal choice.

\begin{remark} \label{rem:main_thms}
  Both Theorems \ref{thm:main_err_tangent_cone} and \ref{thm:no_tancone_Struc} have analogues in the literature for recovering (in the $\ell_2$ norm) a structured signal $\beta^*$ in the linear model \eqref{eq:struc_signal_rec_ind}, where the design matrix $X$ and the noise $\eta$ are typically independent; see for instance \cite{neykov19b, planlasso16}. Indeed, our proofs are structured along the same lines as that in \cite{neykov19b} (which in turn follows ideas from \cite{planlasso16}). However, our linear model in \eqref{eq:lin_dyn_sys_mod} is different, especially due to the temporal correlation between $x_t$ and $\eta_{t},\eta_{t-1},\dots...\eta_{1}$. This leads to additional technical difficulties in the analysis -- as opposed to \cite{neykov19b, planlasso16} we cannot use Gordon's mesh theorem \cite{Gordon88} and need to \revj{instead} use a different set of concentration tools (see Section \ref{subsec:tools}) for controlling the terms in Lemma \ref{lem:ineq_first_ord_cond} (see Section \ref{subsec:proof_main_tancone}). One implication of this is the appearance of the $\gamma_1$ term in our bounds which can often be larger than the $\gamma_2$ term.
\end{remark}

\begin{remark} \label{rem:rel_work_comp}
    %\todo{Add comments on novelty of results: Theorem 2 is new, and Theorem 1 in its present form is also new (since $B \neq A^*$). Add discussion with PGD paper \cite{PGDstruct21}.}
\rev{
As discussed in Section \ref{subsec:rel_work}, existing works for learning structured LDS are typically focused on specific structural assumptions such as sparsity, low-rankness, group sparsity, or low-rank plus sparse matrices, and do not apply for general convex constraints $\calK$.} 

\rev{A work\footnote{We came across this work after the completion of the present paper.} closely related to ours is \cite{PGDstruct21}. For a convex regularizer $R(\cdot)$, it is assumed that $A^* \in \calK$ with $\calK := \set{A: R(A) \leq R(A^*)}$. The PGD method is shown to linearly converge (up to problem-dependent factors) to the `statistical error' $\frac{w(\calT_{\calK,A^*})}{\sqrt{T}}$ provided $T$ is at least of the order $w^2(\calT_{\calK,A^*})$ \cite[Theorem 1]{PGDstruct21}. There are some important points to be noted about this result.}
\begin{itemize}
    \item \revj{The matrix $A^*$ is assumed to be strictly stable, and the extremal eigenvalues ($\kappa_{\max}, \kappa_{\min}$) of the spectral density function of the autocovariance matrix of $(x_t)_{t \geq 1}$ play an important role in the analysis (and appear in the bounds as the ``condition number''). The usage of these quantities was motivated by the work \cite{Basu15}. It is shown that the matrix $\est{A}^{(k)}$ at the $k$th iteration of PGD satisfies the bound}
    \revj{
    \begin{equation*}
        \norm{\est{A}^{(k)} - A^*}_F \leq \rho^{k+1}\norm{\est{A}^{(0)} - A^*}_F + \frac{\xi}{1-\rho}
    \end{equation*}
    where $\xi$ is the statistical error, and $\rho < 1 - \frac{\kappa_{\min}}{2\kappa_{\max}}$ is ensured provided $T$ is suitably large. 
    } 
    
    \item \rev{The proof makes use of an inequality  $\gamma_1(\calS,\norm{\cdot}_F) \leq \gamma_2^2(\calS,\norm{\cdot}_F)$, for a general set $\calS$, which was stated in \cite[Lemma 2.7]{Melnyk16}. We could not validate the proof of this inequality. Without using the inequality, one can verify from the proof of \cite[Theorem 1]{PGDstruct21} that the statistical error therein would then be of the order $$\frac{\gamma_2(\calT_{\calK,A^*} \cap \frobsphere, \norm{\cdot}_F)}{\sqrt{T}} + \frac{\gamma_1(\calT_{\calK,A^*} \cap \frobsphere, \norm{\cdot}_2)}{T}$$ provided $T = \Omega(\gamma_2^2(\calT_{\calK,A^*} \cap \frobsphere), \norm{\cdot}_F)$. This is the same as in Theorem \ref{thm:main_err_tangent_cone} for $B = A^*$. While there are differences in the technical details between our proof of Theorem \ref{thm:main_err_tangent_cone} and that of \cite[Theorem 1]{PGDstruct21}, an important aspect of Theorem \ref{thm:main_err_tangent_cone} is that the bounds are stated in terms of tangent cone structure at $B$ (not necessarily equal to $A^*$) which can be useful in certain situations (we show this in Corollary \ref{cor:lipschitz_reg_examp} for Example $4$).
    }

    \item \rev{Upon inspection of the proof of \cite[Theorem 1]{PGDstruct21}, it seems possible that the analysis applies for general convex sets $\calK$ without any change. However this would have to be verified in detail.}
\end{itemize}
\end{remark}

%--------------------------------------------------
% Instantiating our main results on some examples
%--------------------------------------------------
\subsection{Instantiating our results on examples}
We now recall the examples from Section \ref{sec:intro} for which the bounds in Theorems \ref{thm:main_err_tangent_cone} and \ref{thm:no_tancone_Struc} can be made explicit in terms of $n$. The proofs are detailed in Section \ref{sec:proof_corr}.

\paragraph{Example $1$ ($d$-dimensional subspace).} In this case, observe from the definition of $\tanconeB$ that $\tanconeB = \calK$ for any choice of $B \in \calK$. One can then use standard covering number bounds to bound the $\gamma_{\alpha}$-functionals in Theorem \ref{thm:main_err_tangent_cone}. This leads to the following corollary of Theorem \ref{thm:main_err_tangent_cone} where we take $B = A^*$.
\begin{corollary} \label{cor:subspace}
Let $\calK \subset \matR^{n \times n}$ be a $d$-dimensional subspace. If $T \gtrsim J^4(A^*) \max\set{d, \log^2(1/\delta)}$ for $\delta \in (0,1)$, then it holds with probability at least $1-\delta$ that
\begin{equation*}
    \norm{\est{A} - A^*}_F \lesssim   J(A^*)\left(\frac{\log(1/\delta) + \sqrt{d}}{\sqrt{T}} \right).
\end{equation*}
\end{corollary}
In the unconstrained case where $\calK = \matR^{n \times n}$, so that $d = n^2$, \eqref{eq:Aest_convex} is the OLS estimator. Then Corollary \ref{cor:subspace} states that if $T \gtrsim J^4(A^*) \max\set{n^2, \log^2(1/\delta)}$, we have  
\begin{equation} \label{eq:unconstr_bd}
    \norm{\est{A} - A^*}_F \lesssim J(A^*)\left(\frac{\log(1/\delta) + n}{\sqrt{T}} \right).
\end{equation}
Existing error bounds in the literature for the OLS estimator are typically in the spectral norm, but can of course be converted to the Frobenius norm with an extra factor of $\sqrt{n}$. Indeed, the result of \cite{Jedra20} in \eqref{eq:err_bd_jedra} implies
%Denoting $\Gamma_s(A) = \sum_{k=0}^s A^k (A^k)^{\top}$ for $s \geq 0$, and $\lambda_{\min}(\cdot)$ to be the smallest eigenvalue of a symmetric matrix, they show that w.p at least $1-\delta$, 
%
%\begin{equation*}
%\norm{\est{A} - A^*}_2 \lesssim %\sqrt{\frac{\log(1/\delta) + n}{\lambda_{\min}%(\sum_{s=0}^{T-1} \Gamma_s(A^*))}}
%\end{equation*}
%
%provided $\lambda_{\min}(\sum_{s=0}^{T-1} \Gamma_s(A^*)) \gtrsim J^2(A^*)(\log(1/\delta) + n)$. Moreover, this bound is also optimal up to the dependence on $J(A^*)$. Since $\lambda_{\min}(\sum_{s=0}^{T-1} \Gamma_s(A^*)) \geq T$, this implies
%
\begin{equation} \label{eq:jedra_frob_bd}
  \norm{\est{A} - A^*}_F \leq \sqrt{n} \norm{\est{A} - A^*}_2 \lesssim \sqrt{\frac{n\log(1/\delta) + n^2}{T}}
\end{equation}
provided $T \gtrsim J^2(A^*)(\log(1/\delta) + n)$. The bound in \eqref{eq:jedra_frob_bd} is of the same order as in \eqref{eq:unconstr_bd}, barring the extra $J(A^*)$ term in our bound. Moreover, when $J(A^*)$ is a constant, note that $T$ needs to be at least of the order $n^2$ -- in both \eqref{eq:unconstr_bd} and \eqref{eq:jedra_frob_bd} -- in order to drive the error below a specified threshold. Of course, in case $d \ll n^2$, then the requirement $T \gtrsim d$ in Corollary \ref{cor:subspace} is relatively mild, as one would expect, given that $\calK$ has an intrinsic dimension $d$. 
%
%
%\begin{remark}\label{rem:comp_subsp_lit}
%    \todo{add comparison to [45]}
%\end{remark}
%-------------------------------------------------
% The sparse case, with \ell_1 ball constraint
%-------------------------------------------------
\paragraph{Example $2$ ($\ell_1$ ball).} We now consider the setting where $A^*$ is $k$-sparse, i.e., has at most $k$ non-zero entries for some $1 \leq k \leq n^2$. A standard strategy for recovering sparse matrices (resp. vectors) is to take $\calK$ to be a suitably scaled ball in the $\norm{\cdot}_{1,1}$ (resp. $\ell_1$) norm. We take $\calK := \norm{A^*}_{1,1} \calB_{1,n}$ where
\begin{equation} \label{eq:l11_ball}
    \calB_{1,n} := \set{A \in \matR^{n \times n}: \norm{A}_{1,1} \leq 1}
\end{equation}
so that $A^*$ lies on the boundary of $\calK$. Then the tangent cone $\tanconeAstar$ has the form
\begin{equation*}
    \tanconeAstar := \cl\set{t U: \norm{A^* + U}_{1,1} \leq \norm{\revj{A^{*}}}_{1,1}, \ t \geq 0}.
\end{equation*}
One can bound $w(\tanconeAstar \cap \frobsphere)$ by using existing results in the literature \cite{rudelson08} (see also \cite{stojnic09,lotz14,chandra12,Tropp2015}) -- these results apply for vectors but can be directly invoked in our setting by treating $n \times n$ matrices as vectors in $\matR^{n^2}$. This leads to the same order-wise bound on $\gamma_2(\tanconeAstar \cap \frobsphere, \norm{\cdot}_F)$ due to \eqref{eq:talag_maj_meas_thm} with $\calS = \tanconeAstar \cap \frobsphere$. Furthermore, the $\gamma_1$ functional term can be bounded in terms of $w(\tanconeAstar \cap \frobsphere)$ using Sudakov's minoration inequality \cite[Theorem 7.4.1]{HDPbook}. These considerations lead to the following corollary of Theorem \ref{thm:main_err_tangent_cone}.
\begin{corollary} \label{cor:sparse_example}
    Suppose $A^*$ is $k$-sparse and $\calK = \norm{A^*}_{1,1} \calB_{1,n}$ where $\calB_{1,n}$ is the unit ball defined in \eqref{eq:l11_ball}, and denote $\beta(n,k) := \sqrt{k \log(n^2/k) + k}$. For any $\delta \in (0,1)$, if 
    $$T \gtrsim J^4(A^*) \max\set{\beta^2(n,k), \log^2(1/\delta)}$$
    then with probability at least $1-\delta$ it holds that 
    \begin{equation*}
       \norm{\est{A} - A^*}_F \lesssim J(A^*) \left(\frac{\log(1/\delta) + \beta(n,k)}{\sqrt{T}} + \frac{n \beta(n,k) [1 + \log(\frac{n}{\beta(n,k)})]}{T}\right).
    \end{equation*}
   \end{corollary}
   As a sanity check, note that when $k = n^2$ then $\beta(n,k) = n$ and we recover the statement of Corollary \ref{cor:subspace} with $d = n^2$. This is expected since $A^*$ does not possess any additional structure, and hence the constraint $\calK$ -- the purpose of which is to promote sparse solutions -- does not provide any benefit. The non-trivial sparsity regime is when $k = o(n^2)$. Consider for instance the case where $k \asymp n$. Then, we have $\beta(n,k) \asymp \sqrt{n \log n}$ and Corollary \ref{cor:sparse_example} gives the error bound
   \begin{equation} \label{eq:sparse_rec_bd_1}
       \norm{\est{A} - A^*}_F \lesssim J(A^*) \left(\frac{\log(1/\delta) + \sqrt{n \log n}}{\sqrt{T}} + \frac{(n \log n)^{3/2}}{T} \right)
   \end{equation}
   provided $T \gtrsim J^4(A^*) \max\set{n \log n, \log^2(1/\delta)}$. Notice that while this condition on $T$ implies the error bound in \eqref{eq:sparse_rec_bd_1}, we actually need $T$ to be at least of the order $J(A^*) (n \log n)^{3/2}$ to drive the error in \eqref{eq:sparse_rec_bd_1} below a specified threshold. This ``gap'' is of course due to the term $(n \log n)^{3/2}/T$ in \eqref{eq:sparse_rec_bd_1} -- this term arises from the bound on $\gamma_1(\tanconeAstar \cap \frobsphere, \norm{\cdot}_F)$ within the proof of the corollary and is \revj{suboptimal}. \revj{In general, the bound obtained in the corollary has a suboptimal dependence on $n$, as the optimal (order-wise) rate is expected to be $\sqrt{\frac{k \log(n^2/k)}{T}}$.} Nevertheless, the bound in \eqref{eq:sparse_rec_bd_1} clearly has a milder dependence on $n$ as compared to that obtained for the OLS in \eqref{eq:unconstr_bd}.
   \begin{remark}\label{rem:subopt_sparse_example}
    \revj{If $\calK$ is the $\ell_1$ ball and $A^*$ is $k$-sparse, one could also do a more ``specific'' analysis tailored to this setting by proceeding along similar lines as in the i.i.d case  for the $\ell_1$-constrained LASSO estimator \cite[Theorem 7.13]{wainwright2019high}. Taking $B = A^*$, then by following the steps in the proof of \cite[Theorem 7.13]{wainwright2019high}, we would have that $\est{A} - A^*$ lies in the cone $\calC(\calS,1)$ where $\calS := \text{supp}(A^*)$, and
    \begin{equation*}
    \calC(\calS,r):= \set{U \in \matR^{n \times n}: \norm{(U)_{\calS^c}}_{1,1} \leq r \norm{(U)_{\calS}}_{1,1}}, \quad r > 0. 
\end{equation*}
Subsequently, we would then bound the term $\dotprod{(\est{A} -  A^*) X}{E}$ appearing in Lemma \ref{lem:ineq_first_ord_cond} using H\"older's inequality by $\norm{\est{A} -  A^*}_{1,1} \norm{\vect(X^\top E)}_{\infty}$. Note that $\est{A} -  A^* \in \calC(\calS,1)$ clearly implies $\norm{\est{A} -  A^*}_{1,1} \lesssim \sqrt{k}\norm{\est{A} -  A^*}_F$. By then bounding the term $\norm{\vect(E X^\top )}_{\infty}$ (this term appears, e.g., in \cite{Basu15}, and can also be bounded using arguments in the proof of \cite[Lemma 2]{donnat2026jointlearningnetworklinear}) one would likely improve the bound in Corollary \ref{cor:sparse_example}. However, it is unclear whether we can modify the analysis for general $\calK$ so that instantiating it for this setting leads to (near-) optimal error rates.
}

\revj{As noted in Remark \ref{rem:main_thms}, the proof outline of our results (Theorems \ref{thm:main_err_tangent_cone} and \ref{thm:no_tancone_Struc}) is along the lines of \cite{neykov19b} which focused on the i.i.d regression case for general convex constraints. Interestingly, the ``problematic'' $\gamma_1$ term arising in our analysis does not arise in \cite{neykov19b} -- the local size of $\calK$ is captured entirely by the Gaussian width. Hence instantiating \cite[Theorem 2.3]{neykov19b} for the case where $\beta^*$ is sparse and $\calK$ is the $\ell_1$ ball would lead to the optimal $\ell_2$ error rate.}

   \end{remark}
   \begin{remark}\label{rem:comp_sparseA_lit}
    \revj{The work \cite{Basu15} studies the $\ell_1$ penalized LASSO estimator for normally distributed $(\eta_t)_{t \geq  1}$, when $A^*$ is $k$-sparse and strictly stable. In this case, they obtain obtain the rate (in our notation) $\norm{\est{A} - A^*}_F \lesssim \sqrt{\frac{k \log n}{T}}$ which matches the usual $\ell_2$-rate for estimating sparse vectors in the i.i.d setting. The proof technique follows the usual roadmap for $\ell_1$ regularized least squares problems (see, e.g., \cite[Chapter 7]{wainwright2019high}), and involves showing that $\est{A} - A^*$ lies (w.h.p) in the cone $\calC(\calS,3)$, 
    provided the regularization parameter is chosen suitably. Then, the main effort goes in showing the restricted eigenvalue condition over this cone, using techniques tailored to its structure. Overall, the analysis for $\ell_1$ penalization avoids the appearance of the $\gamma_1(\cdot, \cdot)$ term which leads to a sharp error bound.}
\end{remark}

   %-------------------------
   % Convex regression
   %-------------------------
   \paragraph{Example $3$ (convex regression).}  
   We will now illustrate the use of Theorem \ref{thm:no_tancone_Struc} in the context of convex regression.  Consider the case where $\calK$ is the collection of matrices formed by sampling bivariate convex functions on a square grid of $\Omega = [0,1/\sqrt{2}]^2$, i.e.,
   \begin{align} \label{eq:K_bivariate_convex}
        \calK := \set{A \in \matR^{n \times n}: A_{ij} = f\left(\frac{i-1}{(n-1)\sqrt{2}}, \frac{j-1}{(n-1)\sqrt{2}} \right) \revj{,\; i,j\in[n],} \text{ for some } f:\Omega \rightarrow \matR \text{ convex }}.
   \end{align}
   Then the following corollary of Theorem \ref{thm:no_tancone_Struc} is obtained taking $B = A^*$ with $A^*$ formed by sampling a convex function with a simple (i.e., piecewise affine) structure. 
   \begin{corollary} \label{cor:convex_reg_biv}
     Let $\Omega_1,\dots,\Omega_K$ be convex subsets of $\Omega = [0,1/\sqrt{2}]^2$ which (a) form a partition of $\Omega$, and (b) are such that each $\Omega_i$ is an intersection of at most $s$ pairs of parallel halfspaces. Suppose $A^* \in \calK$ for some convex $f_0$ which is affine on each $\Omega_i$. Then there exists a constant $c > 0$ such that the following is true. For any $\delta \in (0,1)$, if  
     \begin{equation*}
       \frac{n^2}{\log^s n} \geq K c^{s} \ \text{ and } \ T \gtrsim J^4(A^*) \max \set{c^s K \log^s n, \log^2\left(\frac{1}{\delta}\right)},
     \end{equation*}
     then with probability at least $1-\delta$, 
     \begin{equation*}
      \norm{\est{A} - A^*}_F \lesssim J(A^*) \left(\frac{K c^s \log^s n \log(\frac{n^2}{c^s K \log^s n})}{T} + \frac{\log(1/\delta) + c^{s/2} K^{1/2} \log^{s/2} n}{\sqrt{T}}\right).   
     \end{equation*}
   \end{corollary}
   The proof uses metric entropy estimates from \cite[Section 4]{Kur2020ConvexRI} for a collection of bivariate convex functions lying within a ball of radius $x$ around $f_0$. This then leads to bounds for the $\gamma_1$ and $\gamma_2$ functionals appearing in Theorem \ref{thm:no_tancone_Struc}. Interestingly, we find for the specified $A^*$ that  
\begin{align*}
 \gamma_1((\calK - A^*) \cap x\frobball, \norm{\cdot}_2)  &\lesssim  x c^s K \log^s n \log\left(\frac{n^2}{c^s K \log^s n} \right), \\
\gamma_2((\calK-A^*) \cap x\frobball, \norm{\cdot}_2) 
&\lesssim x c^{s/2} K^{1/2} \log^{s/2} n.
\end{align*}
Hence the only dependence on $x$ in the error bound of Theorem \ref{thm:no_tancone_Struc} is through the last term, and since the bound holds for any $x > 0$, we can take the limit $x = 0$.
   
The scaling factor of $\sqrt{2}$ in the definition of $\Omega$ essentially ensures that $\Omega$ lies within the unit $\ell_2$ ball in $\matR^2$, which is needed in order to use the results of \cite{Kur2020ConvexRI}. Note that in order to drive the error below a threshold, it suffices to have 
   \begin{equation*}
       T \gtrsim J^4(A^*) \max \set{c^s K \log^s n \log \left(\frac{n^2}{c^s K \log^s n} \right), \log^2\left(\frac{1}{\delta}\right)}.
   \end{equation*}
   If, e.g., $s$ and $K$ are constants (i.e., $f_0$ has a simple structure), then 
   $$T \gtrsim J^4(A^*) \max\set{\text{polylog}(n), \log^2\left(\frac{1}{\delta}\right)}$$ 
   is sufficient, which is much milder then the requirement for OLS. 

   %-------------------------
   % Lipschitz regression
   %-------------------------
   \paragraph{\rev{Example $4$ (Lipschitz regression).}} 
   \rev{For this example, we illustrate the applicability of Theorem \ref{thm:main_err_tangent_cone} when $B \neq A^*$.  Let $f_1,\dots,f_n$ be unknown Lipschitz functions where $f_i:[0,1] \rightarrow \matR$ and 
   \begin{equation*}
       \abs{f_i(x) - f_i(y)} \leq L_i \abs{x - y}; \ \quad \forall x,y \in [0,1], 
   \end{equation*}
   for some \revjf{$L_i > 0$}. We then assume that each row of $A^*$ is formed by uniformly sampling $f_i$ with a step size $1/T$, i.e., 
   \begin{equation*}
     A^*_{i,j} = f_i(j/T); \quad j=1,\dots,n, 
   \end{equation*}  
   with $T \geq n$. Hence we consider the constraint set $\calK$ to be 
    \begin{equation} \label{eq:lips_reg_set}
       \calK := \set{A \in \matR^{n \times n}:  \abs{A_{i,j} - A_{i,j+1}} \leq \frac{L_i}{T}, \ i \in [n], \ j \in [n-1]}.
    \end{equation}
    This example is essentially motivated by that proposed in \cite[Section 3.1]{neykov19b} which considered $\beta^* \in \matR^n$ to be formed by sampling an unknown Lipschitz function. The techniques we will subsequently employ in the proof of Corollary \ref{cor:lipschitz_reg_examp} below are essentially adaptations of those used in \cite[Corollary 3.4]{neykov19b}. The main idea is to show the existence of $B$ which is a good approximation of $A^*$, and each row of which is formed by sampling a piecewise linear function with few pieces. For ease of exposition, we will outline our final result using $L$ such that $L_i \leq L$ for each $i$.} 
%
%\rev{
\begin{corollary}\label{cor:lipschitz_reg_examp}
    Suppose $A^* \in \calK$ with $\calK$ as in \eqref{eq:lips_reg_set} and let $L_i \leq L$ for each $i$. For any $\delta \in (0,1)$, if
    \begin{equation} \label{eq:T_bd_lipreg_fin}
       T \gtrsim \revj{J^4(A^*) \max \set{\log^2(1/\delta), L^{1/2} n^{3/2} (\log n)^{1/2}, n^{4/3} L^{2/3}, \revjf{n \log n}}}     
    \end{equation}
    then with probability at least $1-\delta$, it holds that  
    \revj{
    \begin{align*}
     \norm{\est{A} - A^*}_F \lesssim J^2(A^*) \Bigg[ \frac{\log(1/\delta)}{\sqrt{T}} &+ \left(\frac{L^{1/3} n (\log n)^{1/3}}{T^{2/3}} \right) \left(\max\set{1,\frac{(n \log n)^{2/3}}{T^{1/3}} }\right) \\ 
     &+ \revjf{\sqrt{\frac{n \log n}{T}} + \frac{(n \log n)^{3/2}}{T}} \Bigg].
    \end{align*}}
\end{corollary}
Let us examine the requirement on $T$ which ensures that $\norm{\est{A} - A^*}_F \leq \epsilon$ for any $\epsilon \in (0,1)$. Apart from the condition on $T$ in \eqref{eq:T_bd_lipreg_fin}, one can verify that we additionally require
\begin{equation*}
  T \gtrsim \revj{J^4(A^*) \max \set{\frac{\log^2(1/\delta)}{\epsilon^2},   \frac{\revjf{ (n \log n)^{3/2}}}{\epsilon^{3/2}}, 
  L^{1/3} \frac{n^{5/3} \log n }{\epsilon}}.}
\end{equation*}
In particular, this means that accurate estimation of $A^*$ is possible with $T = o(n^2)$ samples. 
%}

% Proof outline - main results
%------------------
% Proofs
%--------------------
\section{Proof of main results} \label{sec:proof}
We begin by outlining in Section \ref{subsec:tools} some important concentration tools used in the proof of our main results. Section \ref{subsec:proof_main_tancone} contains the proof of Theorem \ref{thm:main_err_tangent_cone} while Section \ref{subsec:proof_main_no_tancone} contains the proof of Theorem \ref{thm:no_tancone_Struc}.

%-------------------
% Technical tools
%-------------------
\subsection{Technical tools} \label{subsec:tools}
For a set of matrices $\calA$, let us define the terms 
\begin{equation} \label{eq:rad_frob_spec}
d_F(\calA) = \sup_{A \in \calA} \norm{A}_F, \quad d_2(\calA) = \sup_{A \in \calA} \norm{A}_2,
\end{equation}
which can be thought of as other types of complexity measures of the set $\calA$ (the ``radius'' of $\calA$). 

The following result from \cite{krahmer14} provides a concentration bound for the suprema of second order subgaussian chaos processes involving positive semidefinite (p.s.d) matrices. 
\begin{theorem}[\cite{krahmer14}] \label{thm:krahmer_chaos}
Let $\calA$ be a set of matrices and $\xi$ be a vector whose entries are independent, zero-mean, variance $1$, and are $L$-subgaussian random variables. Denote
\begin{align*}
P &= \gamma_2(\calA, \norm{\cdot}_2) [\gamma_2(\calA, \norm{\cdot}_2) + d_F(\calA)] + d_F(\calA) d_2(\calA) \\
V &= d_2(\calA)[\gamma_2(\calA, \norm{\cdot}_2) + d_F(\calA)], \quad \text {and} \quad U = d_2^2(\calA)    
\end{align*}
where $d_2, d_F$ are as in \eqref{eq:rad_frob_spec}. Then there exist constants $c_1,c_2 > 0$ depending only on $L$ such that for any $t > 0$ it holds that 
\begin{equation*}
    \prob\left(\sup_{A \in \calA} \abs{\norm{A \xi}_2^2 - \expec[\norm{A \xi}_2^2]} \geq c_1 P + t \right) \leq 2\exp\left(-c_2 \min \set{\frac{t^2}{V^2}, \frac{t}{U}} \right).
\end{equation*}
\end{theorem}
We will also use the following result from \cite{dirksen15} for bounding the suprema of general second order subgaussian chaos processes, where the matrices are not necessarily p.s.d. The bound on moments in part $1$ is stated on page $15$ in \cite{dirksen15}; part $2$ follows by passing from moment bounds to tail bounds in a standard manner via Markov's inequality, see for example \cite[Lemma A.1]{dirksen15}.
\begin{theorem}[\cite{dirksen15}] \label{thm:dirksen_chaos_conc}
Let $\calA$ be a set of matrices and $\xi$ be a vector whose entries are independent, zero-mean, $1$-subgaussian random variables. For $A \in \calA$, denote $C_A(\xi) := \xi^\top A \xi - \expec[\xi^\top A \xi]$. With $d_2, d_F$ as in \eqref{eq:rad_frob_spec}, the following is true.
\begin{enumerate}
    \item There exists a universal constant $c > 0$ such that for any $p \geq 1$,
    \begin{align*}
        \left(\expec \sup_{A \in \calA} \abs{C_A(\xi)}^p \right)^{1/p} 
        &\leq  c \left(\gamma_1(\calA,\norm{\cdot}_2) + \gamma_2(\calA,\norm{\cdot}_F) +\sqrt{p} d_F(\calA) + p d_2(\calA) \right) 
        \\ 
        &= c \varrho(\calA,p).
    \end{align*}

    \item There exists a universal constant $c' > 0$ (depending on $c$) such that for any $u \geq 1$, 
    \begin{equation*}
        \prob \left(\sup_{A \in \calA} \abs{C_A(\xi)} \geq c' \varrho(\calA,u) \right) \leq e^{-u}.
    \end{equation*}
\end{enumerate}
\end{theorem}

%------------------------
% Proof of theorem
%-----------------------
%
\subsection{Proof of Theorem \ref{thm:main_err_tangent_cone}} \label{subsec:proof_main_tancone}
Let us define the matrices (each of size $n \times T$)
\begin{align*}
    \Xtil = [x_2 \cdots x_{T+1}], \ X = [x_1 \cdots x_T], \ \text{ and } E = [\eta_2 \cdots \eta_{T+1}],
\end{align*}
so that $\Xtil = A^* X + E$ with $x_0 = 0$ and $x_1 = \eta_1$. Then, \eqref{eq:Aest_convex} can be rewritten as 
\begin{equation} \label{eq:Aest_convex_matform}
    \est{A} \in \argmin{A \in \calK} \norm{\Xtil - A X}_F^2.
\end{equation}

\paragraph{Step 1.} Our starting point is the following inequality which follows from first-order optimality conditions for constrained convex programs.
\begin{lemma} \label{lem:ineq_first_ord_cond}
For any $B \in \calK$ the solution $\est{A}$ of \eqref{eq:Aest_convex_matform} satisfies
\begin{equation*}
    \norm{(\est{A} - B)X}_F^2 \leq \dotprod{(\est{A} -  B) X}{E} + \norm{(A^* - B) X}_F \norm{(\est{A} - B) X}_F.
\end{equation*}
\begin{proof}
We first expand $\norm{\Xtil - \est{A} X}_F^2$ as 
\begin{align}
   \norm{\Xtil - \est{A} X}_F^2 
   &= \norm{\Xtil - B X}_F^2 + \norm{(B - \est{A}) X}_F^2 + 2\dotprod{\Xtil - B X}{(B - \est{A}) X} \nonumber \\
   &= \norm{\Xtil - B X}_F^2 - \norm{(B - \est{A}) X}_F^2 + 2\dotprod{\Xtil - \est{A} X}{(B - \est{A}) X} \nonumber \\
   &\leq \norm{\Xtil - B X}_F^2 - \norm{(B - \est{A}) X}_F^2 \label{eq:temp_1}
\end{align}
where the last inequality follows due to $\dotprod{\Xtil - \est{A} X}{(B - \est{A}) X} \leq 0$ by the first order optimality condition for $\est{A}$. 
Using $\Xtil = A^* X + E$, we obtain 
\begin{align*}
 \norm{\Xtil - \est{A} X}_F^2 &= \norm{(\est{A} - A^*) X}_F^2 + \norm{E}_F^2 + 2\dotprod{(A^* - \est{A})X}{E},   \\
 \norm{\Xtil - B X}_F^2 - \norm{(B - \est{A}) X}_F^2 &= \norm{(A^* - B)X}_F^2 + \norm{E}_F^2 \\ &+ 2\dotprod{(A^*-B)X}{E} - \norm{(B-\est{A})X}_F^2.
\end{align*}
Plugging these expressions in \eqref{eq:temp_1} leads to the inequality
\begin{equation}
\norm{(A^* - \est{A})X}_F^2 \leq \norm{(A^* - B) X}_F^2 + 2\dotprod{(\est{A} - B) X}{E} - \norm{(B - \est{A}) X}_F^2. \label{eq:temp_2}
\end{equation}
Expanding the LHS of \eqref{eq:temp_2} leads to the lower bound
\begin{align*}
\norm{(A^* - \est{A})X}_F^2 
&= \norm{(A^* - B)X}_F^2 + \norm{(B - \est{A})X}_F^2 + 2\dotprod{(A^*-B)X}{(B-\est{A})X)} \\
&\geq \norm{(A^* - B)X}_F^2 + \norm{(B - \est{A})X}_F^2 - 2\norm{(A^* - B)X}_F \norm{(B - \est{A})X}_F,    
\end{align*}
and plugging this in \eqref{eq:temp_2} readily leads to the stated inequality in the lemma.
\end{proof}
\end{lemma}
Our next goal is to bound the terms appearing in Lemma \ref{lem:ineq_first_ord_cond}. 
\paragraph{Step 2.}  Consider first the term $\norm{(\est{A} - B) X}_F^2$ which can be written as
\begin{equation*}
  \norm{(\est{A} - B) X}_F^2 = \sum_{t=0}^{T} \norm{(\est{A} - B)x_t}_2^2 = \norm{(I_T \otimes (\est{A} - B)) \vect(X)}_2^2.
\end{equation*}%
One can verify that $\vect(X) = \Gamma \xi$ where 
\begin{equation*}
\Gamma =  \begin{bmatrix}
   I_n & 0 & \hdots & 0\\ 
   A^* & I_n & \hdots & 0\\ 
  \vdots  &  & \ddots & \vdots\\
  (A^*)^{T-1} & \hdots & A^* & I_n
 \end{bmatrix} \in \matR^{Tn \times Tn}, \quad 
 \xi = \begin{bmatrix}
   \eta_1 \\ 
   \eta_2 \\ 
  \vdots \\
  \eta_T
 \end{bmatrix} \in \matR^{Tn},
\end{equation*}
and so, we can write $\norm{(\est{A} - B) X}_F^2 = \norm{(I_T \otimes (\est{A} - B)) \Gamma \xi}_2^2$. As shown in \cite{Jedra20}, we can bound $\norm{\Gamma}_2 \leq J(A^*)$; recall its definition from  \eqref{eq:stab_param}. Now let us define the set
\begin{equation*}
    \tanconeBtil := \set{(I_T \otimes A) \Gamma: A \in \tanconeB \cap \frobsphere};
\end{equation*}
we then clearly have 
$$\frac{(I_T \otimes (\est{A} - B)) \Gamma}{\norm{\est{A} - B}_F} = \left(I_T \otimes \left(\frac{\est{A} - B}{\norm{\est{A} - B}_F} \right)\right) \Gamma \in \tanconeBtil$$ 
since $\est{A} - B \in \tanconeB$. Therefore we can bound the term $\norm{(\est{A} - B) X}_F^2$ as 
\begin{equation} \label{eq:temp_3}
    \left( \inf_{W \in \tanconeBtil} \norm{W \xi}_2^2 \right) \norm{\est{A} - B}_F^2 \leq \norm{(\est{A} - B) X}_F^2 \leq \norm{\est{A} - B}_F^2 \left(\sup_{W \in \tanconeBtil} \norm{W \xi}_2^2 \right).
\end{equation}
The term $\norm{W \xi}_2^2$ is a second order subgaussian chaos involving p.s.d matrices, and we wish to control its infimum and supremum over the set $\tanconeBtil$ in \eqref{eq:temp_3}. This is done using Theorem \ref{thm:krahmer_chaos} and leads to the following lemma.
\begin{lemma} \label{lem:psd_chaos_conc_bounds}
Denote $\Ubar, \Vbar, \Pbar$ as
\begin{align*}
    \Ubar &= J^2(A^*), \quad \Vbar = J(A^*)\left[\gamma_2(\tanconeBtil, \norm{\cdot}_2) + \sqrt{T} J(A^*)\right], \\
    \Pbar &= \gamma_2(\tanconeBtil, \norm{\cdot}_2)\left[\gamma_2(\tanconeBtil, \norm{\cdot}_2) + \sqrt{T} J(A^*) \right] + \sqrt{T} J^2(A^*).
\end{align*}
Then there exist constants $c_1, c_2 > 0$ depending only on $L$ such that for any $t > 0$, it holds with probability at least $1 - 2\exp(-c_2 \min\set{\frac{t^2}{\Vbar^2}, \frac{t}{\Ubar}})$ that 
\begin{enumerate}
    \item $\inf_{W \in \tanconeBtil} \norm{W \xi}_2^2 \geq T - c_1 \Pbar - t$, and

    \item $\sup_{W \in \tanconeBtil} \norm{W \xi}_2^2 \leq T J^2(A^*) + c_1 \Pbar + t$. 
\end{enumerate}
\end{lemma}
\begin{proof}
We will use Theorem \ref{thm:krahmer_chaos} to obtain the stated bounds. Let us first bound the terms $d_F(\tanconeBtil)$, $d_2(\tanconeBtil)$ as follows.
\begin{equation} \label{eq:df_bound_temp1}
    d_F(\tanconeBtil) = \sup_{W \in \tanconeBtil} \norm{W}_F = \sup_{A \in \tanconeB \cap \frobsphere} \norm{(I_T \otimes A) \Gamma}_F \leq J(A^*) \sqrt{T}
\end{equation}
since $\norm{(I_T \otimes A) \Gamma}_F \leq \norm{\Gamma}_2 \norm{I_T \otimes A}_F \leq J(A^*) \sqrt{T}$ for any $A \in \tanconeB \cap \frobsphere$. Furthermore, 
\begin{equation} \label{eq:d2_bound_temp1}
d_2(\tanconeBtil) = \sup_{W \in \tanconeBtil} \norm{W}_2 = \sup_{A \in \tanconeB \cap \frobsphere} \norm{(I_T \otimes A) \Gamma}_2 \leq J(A^*)
\end{equation}
since $\norm{(I_T \otimes A) \Gamma}_2 \leq \norm{A}_2 \norm{\Gamma}_2 \leq \norm{A}_F \norm{\Gamma}_2 \leq J(A^*)$ for all $A \in \tanconeB \cap \frobsphere$. 

Next, we use almost matching bounds on $\expec[\norm{(I_T \otimes A) \Gamma \xi}_2^2]$ holding uniformly over $A \in \tanconeB \cap \frobsphere$ as
\begin{equation*}
    \expec[\norm{(I_T \otimes A) \Gamma \xi}_2^2] = \norm{(I_T \otimes A) \Gamma}_F^2 = \ \ 
    \begin{cases}
       \leq T J^2(A^*) \\
       \geq T \norm{A}_F^2 = T
    \end{cases}
\end{equation*}
where for the lower bound we used the fact that $A$ appears $T$ times within $(I_T \otimes A) \Gamma$. Finally, one can readily verify that the terms $\Pbar, \Ubar, \Vbar$ are (resp.) upper bounds for $P, U$ and $V$, the latter terms defined in Theorem \ref{thm:krahmer_chaos}. This concludes the proof. 
\end{proof}
Hence the event in Lemma \ref{lem:psd_chaos_conc_bounds} implies the bounds 
\begin{align}
    \norm{(\est{A} - B) X}_F^2 &\geq \norm{\est{A} - B}_F^2 (T - c_1 \Pbar - t), \label{eq:temp_bd_11} \\
    \norm{(\est{A} - B) X}_F &\leq \norm{\est{A} - B}_F (T J^2(A^*) + c_1 \Pbar + t)^{1/2}, \label{eq:temp_bd_12} \\
     \text { and } \ \norm{(A^* - B) X}_F &\leq \norm{A^* - B}_F (T J^2(A^*) + c_1 \Pbar + t)^{1/2}. \label{eq:temp_bd_13}
\end{align}
\paragraph{Step 3.} Our goal now is to control the term $\dotprod{(\est{A} -  B) X}{E}$. We can first bound it as 
\begin{equation*}
    \dotprod{(\est{A} -  B) X}{E} = \norm{\est{A} -  B}_F \left\langle\left(\frac{\est{A} -  B}{\norm{\est{A} -  B}_F}\right) X, E \right \rangle \leq  \norm{\est{A} -  B}_F \left(\sup_{A \in \tanconeB \cap \frobsphere} \dotprod{AX}{E} \right).
\end{equation*}
It remains to control the supremum term, which in fact is the supremum of a second order subgaussian chaos involving matrices that are not necessarily p.s.d (as will be seen in the proof below). This is achieved via Theorem \ref{thm:dirksen_chaos_conc} and leads to the following lemma.
\begin{lemma} \label{lem:gen_chaos_conc_bds}
For any $u \geq 1$, we have with probability at least $1 - \exp(-u)$, 
\begin{equation*}
    \sup_{A \in \tanconeB \cap \frobsphere} \abs{\dotprod{AX}{E}} \leq 
    c_3 \left(u J(A^*) + \sqrt{u T} J(A^*) + \gamma_1(\tanconeBtil, \norm{\cdot}_2) + \gamma_2(\tanconeBtil, \norm{\cdot}_F) \right)
\end{equation*}
for some constant $c_3 > 0$ depending only on $L$.
\end{lemma}
\begin{proof}
We begin by rewriting $\dotprod{AX}{E}$ as 
\begin{align*}
    \dotprod{AX}{E} = \dotprod{A^\top E}{X} = \dotprod{(I_T \otimes A^\top)\xitil}{\Gamma \xi} = \xitil^\top ((I_T \otimes A) \Gamma) \xi \ \text{ where } \ \xitil = \vect(E) = \begin{bmatrix}
   \eta_2 \\ 
   \eta_3 \\ 
  \vdots \\
  \eta_{T+1}
 \end{bmatrix} \in \matR^{Tn}.
\end{align*}
%
%\denis{$E$ is composed by $\eta_j$ for $j=1,\dots,T$, and the same for $\xi$. In fact, $\xi=\vect(E)$. Why $\xitil$ has appeared?} \hemant{Yes indeed, notice that $\xitil$ is a ``shifted'' version of $\xi$ as the indexing for $\xi$ goes from $1$ to $T$. We need both $\xi$ and $\xitil$ to write down $\dotprod{AX}{E}$, at least initially as in above equation.} \denis{I still do not understand: $\langle A^\top E,X\rangle = \langle (I_T \otimes A^\top)\xi,\Gamma\xi\rangle$? There is no $\eta_{T+1}$ in $E$ or $X$? Why it is shifted?}
%
Denoting $\eta \in \matR^{(T+1)n}$ to be the vector formed by stacking $\eta_1,\dots,\eta_{T+1}$, we can further simplify $\xitil^\top ((I_T \otimes A) \Gamma) \xi$ as 
\begin{equation*}
    \xitil^\top ((I_T \otimes A) \Gamma) \xi 
    = \eta^\top \underbrace{\begin{bmatrix}
   0 &  0 \\ 
   (I_T \otimes A)\Gamma &  0
 \end{bmatrix}}_{M_A} \eta 
 = \eta^\top M_{A} \eta. 
\end{equation*}
This in particular implies $\expec[\dotprod{AX}{E}] = \expec[\eta^\top M_{A} \eta] = 0$. Denoting the set 
\begin{equation*}
    \mathcal{M} = \set{M_A : A \in \tanconeB \cap \frobsphere}
\end{equation*}
we obtain
$$\sup_{A \in \tanconeB \cap \frobsphere} \dotprod{AX}{E} = \sup_{M_A \in \mathcal{M}} \eta^\top M_{A} \eta.$$
We now use Theorem \ref{thm:dirksen_chaos_conc} to control this suprema. To this end, we observe that
\begin{align*}
d_F(\mathcal{M}) &= \sup_{M_A \in \mathcal{M}} \norm{M_A}_F = \sup_{A \in \tanconeB \cap \frobsphere} \norm{(I_T \otimes A)\Gamma}_F \leq \sqrt{T} J(A^*), \tag{\text{using \eqref{eq:df_bound_temp1}}} \\
d_2(\mathcal{M}) &= \sup_{M_A \in \mathcal{M}} \norm{M_A}_2 = \sup_{A \in \tanconeB \cap \frobsphere} \norm{(I_T \otimes A)\Gamma}_2 \leq J(A^*) \tag{using \eqref{eq:d2_bound_temp1}}.
\end{align*}
It is easy to see that\footnote{\revo{Since $\tanconeBtil \mapsto \mathcal{M}$ is $1-1$, and $\norm{\revj{M_{A_1} - M_{A_2}}} = \norm{\revj{(I_T \otimes A_1)\Gamma - (I_T \otimes A_2)\Gamma}}$ when $\norm{\cdot}$ is either the spectral or Frobenius norm.}} $\gamma_{\alpha}(\mathcal{M}, \norm{\cdot}) = \gamma_{\alpha}(\tanconeBtil, \norm{\cdot})$ \revo{for} $\norm{\cdot}$ \revo{corresponding to either the spectral or Frobenius norm}. Then by using Theorem \ref{thm:dirksen_chaos_conc} with $\xi$ therein corresponding to $\frac{\eta}{L}$, we readily arrive at the statement of the lemma.
\end{proof}
The event in Lemma \ref{lem:gen_chaos_conc_bds} implies the bound
\begin{align}
    \dotprod{(\est{A} -  B) X}{E} 
    &\leq  c_3 \norm{\est{A} -  B}_F \left(u J(A^*) + \sqrt{u T} J(A^*) + \gamma_1(\tanconeBtil, \norm{\cdot}_2) + \gamma_2(\tanconeBtil, \norm{\cdot}_F) \right) \nonumber \\
    &= c_3  \varrho(\tanconeBtil, u)  \norm{\est{A} -  B}_F\label{eq:temp_bd_14}
\end{align}
with $\varrho(\cdot, \cdot)$ as defined in Theorem \ref{thm:dirksen_chaos_conc}. 
\paragraph{Step 4: Putting it together.} Using the results from \eqref{eq:temp_bd_11}, \eqref{eq:temp_bd_12}, \eqref{eq:temp_bd_13}, \eqref{eq:temp_bd_14} in Lemma \ref{lem:ineq_first_ord_cond}, we have with probability at least $1 - 2\exp(-c_2 \min\set{\frac{t^2}{\Vbar^2}, \frac{t}{\Ubar}}) - \exp(-u)$  that 
\begin{align}
    \norm{\est{A} - B}_F^2 (T - c_1 \Pbar - t) 
    &\leq c_3 \varrho(\tanconeBtil, u) \norm{\est{A} - B}_F  + \norm{A^* - B}_F \norm{\est{A} - B}_F \left(T J^2(A^*) + c_1 \Pbar + t \right) \nonumber \\
    \iff \norm{\est{A} - B}_F &\leq c_3 \frac{\varrho(\tanconeBtil, u)}{T - c_1 \Pbar - t} + \norm{A^* - B}_F \left(\frac{T J^2(A^*) + c_1 \Pbar + t}{T - c_1 \Pbar - t} \right). \label{eq:est_bd_put_together_1}
\end{align}
Our aim is to simplify the above bounds and also showcase the dependency on the local tangent cone $\tanconeB$. To this end, the following claim is useful.
\begin{claim} \label{claim:gamma_func_bds}
    For any $\alpha > 0$, \revo{there is a constant $c_{\alpha} > 0$ such that}
    %\hemant{$\alpha > 0$? check!}
    %
    %
    \begin{align*}
        \gamma_{\alpha}(\tanconeBtil, \norm{\cdot}_2) 
        &\leq \revo{c_\alpha} J(A^*) \gamma_{\alpha}(\tanconeB \cap \frobsphere, \norm{\cdot}_2), \\
         \gamma_{\alpha}(\tanconeBtil, \norm{\cdot}_F) &\leq \revo{c_\alpha} \sqrt{T} J(A^*) \gamma_{\alpha}(\tanconeB \cap \frobsphere, \norm{\cdot}_F).
    \end{align*}
\end{claim}
\begin{proof}
    For any $\revj{A_1,A_2} \in \tanconeB \cap \frobsphere$, consider the matrices $\revj{W_1} = (\revj{I_T} \otimes \revj{A_1}) \Gamma$ and $\revj{W_2} = (\revj{I_T} \otimes \revj{A_2}) \Gamma$. Clearly, $\norm{\revj{W_1 - W_2}}_2 \leq J(A^*) \norm{\revj{A_1 - A_2}}_2$ and $\norm{\revj{W_1 - W_2}}_F \leq \sqrt{T} J(A^*) \norm{\revj{A_1 - A_2}}_F$. \revo{Since the mapping $\tanconeB \cap \frobsphere \mapsto \tanconeBtil$ is onto,} this then readily implies the stated bounds using the \revo{property} of $\gamma_{\alpha}$ functionals \revo{for Lipschitz and onto maps (stated in Section \ref{subsec:prelim})}.
\end{proof}
Claim \ref{claim:gamma_func_bds} leads to the following bounds on the terms $\Pbar, \Vbar$ (\revo{for some constant $C \geq 1$}).
\begin{align*}
    \Pbar &\leq \Pbar_1 := \revo{C} J^2(A^*) \left(\gamma_{2}^2(\tanconeB \cap \frobsphere, \norm{\cdot}_2) +\sqrt{T} \gamma_{2}(\tanconeB \cap \frobsphere, \norm{\cdot}_2) + \sqrt{T} \right) \\
    \Vbar &\leq \Vbar_1 := \revo{C} J^2(A^*) \left(\gamma_{2}(\tanconeB \cap \frobsphere, \norm{\cdot}_2) +\sqrt{T} \right)
\end{align*}
Furthermore, for $1 \leq u \leq T$ we can bound $\varrho(\tanconeBtil, u)$ as (\revo{for some constant $C \geq 1$})
\begin{equation*}
    \varrho(\tanconeBtil, u) \leq \revo{C} J(A^*) \left(2\sqrt{uT} + \gamma_1(\tanconeB \cap \frobsphere, \norm{\cdot}_2) + \sqrt{T} \gamma_2(\tanconeB \cap \frobsphere, \norm{\cdot}_F) \right). 
\end{equation*}
The above considerations lead to the following simplification of \eqref{eq:est_bd_put_together_1},
\begin{align}
    \norm{\est{A} - B}_F \leq c_3 &\frac{J(A^*) \left(2\sqrt{uT} + \gamma_1(\tanconeB \cap \frobsphere, \norm{\cdot}_2) + \sqrt{T} \gamma_2(\tanconeB \cap \frobsphere, \norm{\cdot}_F) \right)}{T - c_1 \Pbar_1 - t} \nonumber \\
    &+ \norm{A^* - B}_F \frac{T J^2(A^*) + c_1 \Pbar_1 + t}{T - c_1 \Pbar_1 - t} \label{eq:est_bd_put_together_2}
\end{align}
which holds with probability at least $1 - 2\exp(-c_2 \min\{\frac{t^2}{\Vbar_1^2}, \frac{t}{\Ubar_1}\}) - \exp(-u)$.

Now choosing $t = \Vbar_1 \sqrt{u} $, note that
\begin{equation*}
    \min\set{\frac{t^2}{\Vbar_1^2}, \frac{t}{\Ubar_1}} = \min\set{u, \sqrt{u} \frac{\Vbar_1}{\Ubar}} \geq \min(u, \sqrt{u}) = \sqrt{u}
\end{equation*}
where the inequality holds since $\Vbar_1 \geq \Ubar$, and the final equality uses the condition $u \geq 1$. 
Also note that \eqref{eq:est_bd_put_together_2} holds provided $c_1 \Pbar_1 + t < T$. Using the condition $u \geq 1$, it is then easily verified that 
$$c_1 \Pbar_1 + t \leq  c_1' J^2(A^*) \left(\gamma_2^2(\tanconeB \cap \frobsphere, \norm{\cdot}_2) + \sqrt{T} \gamma_2(\tanconeB \cap \frobsphere, \norm{\cdot}_2) + \sqrt{uT} \right)$$
for some constant $c_1' > 0$ depending on $c_1$. 
%\denis{In $\overline{P}_1$ there is $\sqrt{T}$, but in the new estimate there is only $\sqrt{uT}$, which came from $t$, but since $u\leq T$, then $\sqrt{uT}$ is of order $T$ and not $\sqrt{T}$?} 
So it suffices to ensure that $c_1 \Pbar_1 + t \leq T/2$ for which a sufficient condition is  
\begin{align} \label{eq:temp_4}
    J^2(A^*) \left(\frac{\gamma_2^2(\tanconeB \cap \frobsphere, \norm{\cdot}_2)}{T} + \frac{\gamma_2(\tanconeB \cap \frobsphere, \norm{\cdot}_2)}{\sqrt{T}} + \sqrt{\frac{u}{T}} \right) \leq \frac{1}{2 c_1'}.
\end{align}
Since $J(A^*) \geq 1$, hence \eqref{eq:temp_4} holds provided 
$$T \geq c_1'' J^4(A^*) \max\set{\gamma_2^2(\tanconeB \cap \frobsphere, \norm{\cdot}_2), u}$$
for a suitably large constant $c_1'' > 1$ (depending on $c_1'$). Observe that the above condition implies $u \leq T$.

The above considerations are summarized in the form of the following theorem. 
\begin{theorem}
    There exist constants $C_1, C_2, C_3, C_4 > 0$ depending only on $L$ such that the following is true. For any $u \geq 1$ and $B \in \calK$, suppose that 
    $$T \geq C_1 J^4(A^*) \max\set{\gamma_2^2(\tanconeB \cap \frobsphere, \norm{\cdot}_2), u}.$$ 
    Then with probability at least $1 - C_2\exp(-C_3 \sqrt{u})$, the estimate $\est{A}$ in \eqref{eq:Aest_convex_matform} satisfies
    \begin{equation*}
      \norm{\est{A} - B}_F \leq C_4 J(A^*) \left(\sqrt{\frac{u}{T}} + \frac{\gamma_1(\tanconeB \cap \frobsphere, \norm{\cdot}_2)}{T} + \frac{\gamma_2(\tanconeB \cap \frobsphere, \norm{\cdot}_F)}{\sqrt{T}} \right) 
    + 3 J^2(A^*) \norm{A^* - B}_F.    
    \end{equation*}
\end{theorem}
The statement of Theorem \ref{thm:main_err_tangent_cone} is obtained by choosing $u = \frac{1}{C_3^2} \log^2(\frac{C_2}{\delta})$ for any $\delta \in (0,1)$, and by the triangle inequality
\begin{align*}
   \norm{\est{A} - A^*}_F &\leq  \norm{\est{A} - B}_F +  \norm{A^* - B}_F \\
   &\leq C_4 J(A^*) \left(\sqrt{\frac{u}{T}} + \frac{\gamma_1(\tanconeB \cap \frobsphere, \norm{\cdot}_2)}{T} + \frac{\gamma_2(\tanconeB \cap \frobsphere, \norm{\cdot}_F)}{\sqrt{T}} \right) 
    + 4 J^2(A^*) \norm{A^* - B}_F.
\end{align*}
This completes the proof.

%--------------------------------------------------
% Proof of Theorem 2 (no tangent cone structure)
%
\subsection{Proof of Theorem \ref{thm:no_tancone_Struc}} \label{subsec:proof_main_no_tancone}
We will follow similar notation as in the proof of Theorem \ref{thm:main_err_tangent_cone}. Our starting point is inequality in Lemma \ref{lem:ineq_first_ord_cond} which we recall below.
\begin{equation*}
        \norm{(\est{A} - B)X}_F^2 \leq \dotprod{(\est{A} -  B) X}{E} + \norm{(A^* - B) X}_F \norm{(\est{A} - B) X}_F.
\end{equation*}
Suppose $\norm{\est{A}-B}_F \geq x$ since otherwise the statement of the theorem holds trivially\footnote{\revjf{Indeed, if $\norm{\est{A} - B}_F \leq x$ then $\norm{\est{A} - A^*}_F \leq \norm{A^*- B}_F + x$, which is subsumed by the bound in Theorem \ref{thm:no_tancone_Struc}.}}. Denote $a=\frac{x}{\norm{\est{A}-B}_F} \leq 1$. Since $\calK$ is convex, hence $U= a \est{A} + (1-a) B \in \calK$ which implies $a(\est{A} - B) = U - B$.

Denoting $\calKtil = \calK - B$ from now on, note that $\calKtil$ is a star-shaped set\footnote{A set $\calX$ is star shaped if $\lambda \calX \subset \calX$ for any $\lambda \in [0,1]$. Any convex set containing the origin is star-shaped.} since $\calK$ is convex and $\calKtil$ contains the origin. We would like to obtain an analogue of \cite[Lemma 4.4]{planlasso16} to our setting where we lower bound
\begin{align*}
    \inf_{M \in \calKtil, \norm{M}_F \geq x} \norm{MX}_F = \inf_{M \in \calKtil, \norm{M}_F \geq x} \norm{(I_T \otimes M) \Gamma \xi}_2
\end{align*}
Denoting $\Mtil = x\frac{M}{\norm{M}_F}$ for any $M \in \calKtil$ such that $\norm{M}_F \geq x$, clearly $\Mtil \in \calKtil$ since $\calKtil$ is star-shaped. This implies 
\begin{equation} \label{eq:no_tancone_tmp1}
    \inf_{M \in \calKtil, \norm{M}_F \geq x} \frac{\norm{(I_T \otimes M) \Gamma \xi}_2}{\norm{M}_F} = \inf_{\Mtil \in \calKtil \cap x \frobsphere} \frac{\norm{(I_T \otimes \Mtil) \Gamma \xi}_2}{x}.
\end{equation}
Now we can bound $\expec[\norm{(I_T \otimes \Mtil) \Gamma \xi}_2^2]$ uniformly over $\Mtil \in \calKtil \cap x \frobsphere$ as
\begin{equation*}
\expec[\norm{(I_T \otimes \Mtil) \Gamma \xi}_2^2] = \norm{(I_T \otimes \Mtil) \Gamma}_F^2 \ 
    \begin{cases}
       \leq x^2 T J^2(A^*) \\
       \geq T \norm{\Mtil}_F^2 = T x^2.
    \end{cases}
\end{equation*}
Furthermore, denoting $\calS_x := \set{(I_T \otimes \Mtil) \Gamma: \Mtil \in \calKtil \cap x \frobsphere}$, it is easy to verify that 
\begin{align*}
d_F(\calS_x) &\leq J(A^*) x \sqrt{T}, \quad d_2(\calS_x) \leq J(A^*) x.
\end{align*}
This implies the following bounds on the terms $P,U$ and $V$.
\begin{align*}
P &\leq \gamma_2(\calS_x, \norm{\cdot}_2)\left(\gamma_2(\calS_x, \norm{\cdot}_2 ) + x J(A^*) \sqrt{T} \right) + x^2 J^2(A^*) \sqrt{T} =: \Pbar, \\
V &\leq x J(A^*) \left(\gamma_2(\calS_x, \norm{\cdot}_2) + x J(A^*)\sqrt{T} \right) =: \Vbar, \\
U &= d_2^2(\calS_x) \leq x^2 J^2(A^*) =: \Ubar.
\end{align*}
Note that \eqref{eq:no_tancone_tmp1} also holds with $\inf$ replaced by $\sup$ on both sides. Hence using Theorem \ref{thm:krahmer_chaos}, it follows that with probability at least $1 - 2\exp(-c_2 \min\set{\frac{t^2}{\Vbar^2}, \frac{t}{\Ubar}})$, 
\begin{enumerate}
    \item $\inf_{M \in \calKtil, \norm{M}_F \geq x} \frac{\norm{(I_T \otimes M) \Gamma \xi}_2}{\norm{M}_F} \geq \frac{(Tx^2 - c_1 \Pbar - t)^{1/2}}{x}$, and

    \item $\sup_{M \in \calKtil, \norm{M}_F \geq x} \frac{\norm{(I_T \otimes M) \Gamma \xi}_2}{\norm{M}_F} \leq \frac{(Tx^2 J^2(A^*) + c_1 \Pbar + t)^{1/2}}{x}$,
\end{enumerate}
which in turn implies the bounds 
\begin{align}
    \norm{(\est{A} - B) X}_F^2 &\geq \frac{\norm{\est{A} - B}_F^2}{x^2} (Tx^2 - c_1 \Pbar - t), \label{eq:no_tancone_tmp21} \\
    \norm{(\est{A} - B) X}_F &\leq \frac{\norm{\est{A} - B}_F}{x} (T x^2 J^2(A^*) + c_1 \Pbar + t)^{1/2}.  \label{eq:no_tancone_tmp22} 
     %\text { and } \ \norm{(A^* - B) X}_F &\leq \frac{\norm{A^* - B}_F}{x} (T x^2 J^2(A^*) + c_1 \Pbar + t)^{1/2}. \label{eq:no_tancone_tmp23}
\end{align}
\revjf{We next upper bound the term $\norm{(A^* - B) X}_F$ using the Hanson-Wright inequality \cite{rudelson2013}. Note that $\norm{(A^* - B) X}_F^2 = \norm{(I_T \otimes (A^*-B))\Gamma\xi}_2^2 = \xi^\top M \xi$ where $M = \Gamma^\top (I_T \otimes[(A^*-B)^\top(A^*-B)]) \Gamma$. We then have that $$\expec[\xi^\top M \xi] = \norm{(I_T \otimes (A^*-B)) \Gamma}_F^2 \leq J^2(A^*) T\norm{A^*-B}_F^2.$$
Moreover, it is easily verified that $\norm{M}_2 \leq J^2(A^*) \norm{A^*-B}_F^2$ and $\norm{M}_F^2 \leq J^4(A^*) T \norm{A^*-B}_F^4$. Then the Hanson-Wright inequality implies for all $t > 0$
\begin{align*}
    \prob\left(\abs{\xi^\top M \xi -\expec[\xi^\top M \xi]} \geq t \right) \leq 2\exp\left(-c\min\set{\frac{t^2}{L^4 J^4(A^*) T\norm{A^*-B}_F^4}, \frac{t}{L^2 J^2(A^*) \norm{A^*-B}_F^2}}\right)
\end{align*}
for a constant $c > 0$. Choosing $t = C_1 L^2 J^2(A^*) \norm{A^*-B}_F^2 \log(1/\delta)$ for $\delta \in (0,1)$ and suitably large constant $C_1 > 0$, it follows after some simplifications that there exist constants $C_2, C_3, C_4 > 0$ such that with probability at least $1-\delta$,
\begin{align}\label{eq:no_tancone_tmp23}
   \norm{(A^* - B) X}_F \leq C_2 J(A^*) \sqrt{T} \norm{A^*-B}_F 
\end{align}
provided $T \geq C_3 \log^2(C_4/\delta)$.
}

We now have to bound $\dotprod{(\est{A}-B)X}{E}$. Denoting $\delta = \norm{\est{A} - B}_F$, the fact that $\calKtil$ is star-shaped implies $\delta^{-1} \calKtil \subset x^{-1} \calKtil$ (since $\delta^{-1} \leq x^{-1}$). This then implies 
\begin{align*}
    \dotprod{\frac{(\est{A}-B)X}{\norm{\est{A}-B}_F}}{E} \leq \sup_{U \in \delta^{-1}\calKtil \cap \frobball} \dotprod{UX}{E} \leq \sup_{U \in x^{-1} \calKtil \cap\frobball} \dotprod{UX}{E} = \frac{1}{x} \sup_{U \in \calKtil \cap x\frobball} \dotprod{UX}{E}
\end{align*}
which means that 
\begin{equation*}
\dotprod{(\est{A}-B)X}{E} \leq \frac{\norm{\est{A}-B}_F}{x} \sup_{U \in \calKtil \cap x\frobball} \dotprod{UX}{E}.
\end{equation*}
Now as in the proof of Theorem \ref{thm:main_err_tangent_cone}, we can write 
\begin{align*}
    \dotprod{UX}{E} = \eta^\top \underbrace{\begin{bmatrix}
   0 &  0 \\ 
   (I_T \otimes U)\Gamma &  0
 \end{bmatrix}}_{M_U} \eta
 = \eta^{\top} M_U \eta
\end{align*}
with $\expec[\dotprod{UX}{E}] = 0$. Denoting $\calU_x := \set{M_U : U \in \calKtil \cap x\frobball}$, we then obtain 
\begin{equation*}
    \sup_{U \in \calKtil \cap x\frobball} \dotprod{UX}{E} = \sup_{M_U \in \calU_x} \eta^{\top} M_U \eta. 
\end{equation*}
It can be checked that $d_F(\calU_x) \leq x J(A^*) \sqrt{T}$ and $d_2(\calU_x) \leq J(A^*) x$. Moreover, denoting $\calStil_x:= \set{(I_T \otimes U)\Gamma: U \in \calKtil \cap x\frobball}$, we have that $\gamma_{\alpha}(\calU_x,\norm{\cdot}) = \gamma_{\alpha}(\calStil_x,\norm{\cdot})$ for $\alpha > 0$ and \revo{norm $\norm{\cdot}$ corresponding to either the spectral or Frobenius norm (via arguments similar to that seen earlier)}. 

Then by invoking Theorem \ref{thm:dirksen_chaos_conc}, it follows that for $u \geq 1$, it holds with probability at least $1-\exp(-u)$ that 
\begin{align}
\sup_{U \in \calKtil \cap x\frobball} \abs{\dotprod{UX}{E}} &\leq c_3 \left( \gamma_1(\calStil_x,\norm{\cdot}_2) + \gamma_2(\calStil_x,\norm{\cdot}_F) + \sqrt{uT} x J(A^*) + ux J(A^*) \right) \nonumber \\
&=: \revo{c_3} \rhobar(\calStil_X,x,u), \nonumber\\
\implies \dotprod{(\est{A}-B)X}{E} &\leq c_3\frac{\norm{\est{A}-B}_F}{x} \rhobar(\calStil_x,x,u). \label{eq:no_tancone_tmp3}
\end{align}
Hence using \eqref{eq:no_tancone_tmp21}, \eqref{eq:no_tancone_tmp22},   \eqref{eq:no_tancone_tmp23} and \eqref{eq:no_tancone_tmp3}, it holds with probability at least $1 - \revjf{\delta}$ $- 2\exp(-c_2 \min\set{\frac{t^2}{\Vbar^2}, \frac{t}{\Ubar}})$ $- \exp(-u)$ (for $u \geq 1$)  that 
\begin{align}
\revjf{\norm{\est{A} - B}_F} &\leq \revjf{c_3 \frac{x \rhobar(\calStil_x, x, u)}{Tx^2 - c_1 \Pbar - t} + c_4 \norm{A^* - B}_F \left(\frac{x\sqrt{T} J(A^*)(T x^2 J^2(A^*) + c_1 \Pbar + t)^{1/2}}{Tx^2 - c_1 \Pbar - t} \right)} \nonumber \\ 
&\leq c_3 \frac{x \rhobar(\calStil_x, x, u)}{Tx^2 - c_1 \Pbar - t} + \revjf{c_4} \norm{A^* - B}_F \left(\frac{T x^2 J^2(A^*) + c_1 \Pbar + t}{Tx^2 - c_1 \Pbar - t} \right). \label{eq:no_tancone_tmp4}
\end{align}
It remains to simplify the expressions of the terms involved in our bounds. To this end, we obtain analogous to Claim \ref{claim:gamma_func_bds} that for any $\alpha > 0$, \revo{there is a constant $c_\alpha > 0$ such that}
\begin{align*}
\gamma_{\alpha}(\calS_x,\norm{\cdot}_2) &\leq \revo{c_\alpha} \gamma_{\alpha}(\calKtil \cap x\frobsphere,\norm{\cdot}_2) J(A^*), \\
\gamma_{\alpha}(\calS_x,\norm{\cdot}_F) &\leq \revo{c_\alpha \sqrt{T} J(A^*) \gamma_{\alpha}(\calKtil \cap x\frobsphere,\norm{\cdot}_F)}, \\
\gamma_{\alpha}(\calStil_x,\norm{\cdot}_2) &\leq \revo{c_\alpha} \gamma_{\alpha}(\calKtil \cap x\frobball,\norm{\cdot}_2) J(A^*), \\
\gamma_{\alpha}(\calStil_x,\norm{\cdot}_F) &\leq \revo{c_\alpha \sqrt{T} J(A^*) \gamma_{\alpha}(\calKtil \cap x\frobball,\norm{\cdot}_F)}.
\end{align*}
Furthermore, we can bound $\Pbar, \Vbar$ as (\revo{for some constant $c > 0$})
\begin{align*}
\Pbar &\leq \revo{c} \gamma_2(\calKtil \cap x\frobsphere,\norm{\cdot}_2) J(A^*) \left(\gamma_2(\calKtil\cap x\frobsphere,\norm{\cdot}_2) J(A^*) + x J(A^*) \sqrt{T} \right) + x^2 J^2(A^*) \sqrt{T} = : \Pbar_1, \\
\Vbar &\leq \revo{c} x J^2(A^*) \left(\gamma_2(\calKtil\cap x\frobsphere,\norm{\cdot}_2) + x\sqrt{T} \right) =: \Vbar_1,
\end{align*}
and since $1 \leq u \leq T$, we also obtain 
\begin{align*}
\rhobar(\calStil_x, x, u) 
&\leq \revo{c} \gamma_1(\calKtil \cap x\frobball, \norm{\cdot}_2) J(A^*) + \gamma_2(\calKtil \cap x\frobball,\norm{\cdot}_F)\sqrt{T} J(A^*) + 2\sqrt{uT} xJ(A^*) \\
&= \revo{c} J(A^*) \left(\gamma_1(\calKtil \cap x\frobball, \norm{\cdot}_2) + \sqrt{T} \gamma_2(\calKtil \cap x\frobball,\norm{\cdot}_F) + 2x\sqrt{uT}\right).
\end{align*}
Hence when $1 \leq u \leq T$, we can simplify \eqref{eq:no_tancone_tmp4} as 
\begin{align*}
    \norm{\est{A} - B}_F &\leq c_3 \frac{x J(A^*) \left(2x\sqrt{uT} + \gamma_1(\calKtil \cap x\frobball, \norm{\cdot}_2) + \sqrt{T} \gamma_2(\calKtil \cap x\frobball, \norm{\cdot}_F) \right)}{Tx^2 - c_1 \revo{\Pbar_1} - t} 
    \\
    &+ \revjf{c_4} \norm{A^* - B}_F \left(\frac{T x^2 J^2(A^*) + c_1 \revo{\Pbar_1} + t}{Tx^2 - c_1 \revo{\Pbar_1} - t} \right), 
\end{align*}
which is analogous to \eqref{eq:est_bd_put_together_2}. The remaining steps of the proof follow the same calculations as in the proof of Theorem \ref{thm:main_err_tangent_cone} (details omitted) which leads us to the statement of the theorem. We only remark that the condition on $T$ follows from the fact $\gamma_2(\calKtil \cap x\frobsphere, \norm{\cdot}_2) \lesssim \gamma_2(\calKtil \cap x\frobball, \norm{\cdot}_2)$ since $\frobsphere \subset \frobball$.

% Proof outline - corollaries
%------------------------
% Proofs of corollaries
%------------------------
\section{Proof of corollaries for examples} \label{sec:proof_corr}
%
%---------------------------------------------
% Proof of Corollary for d-dim subspace
%---------------------------------------------
\subsection{Proof of Corollary \ref{cor:subspace}}
Since $\tanconeB = \calK$ for any $B \in \calK$, we choose  $B = A^*$. We will bound the terms $\gamma_2(\calK \cap \frobsphere, \norm{\cdot}_F)$, $\gamma_1(\calK \cap \frobsphere, \norm{\cdot}_2)$ using \eqref{eq:gamma_bound}. We first employ the simplified bound 
\begin{equation*}
    \gamma_1(\calK \cap \frobsphere, \norm{\cdot}_2) \leq \gamma_1(\calK \cap \frobsphere, \norm{\cdot}_F)
\end{equation*}
since $\norm{\cdot}_2 \leq \norm{\cdot}_F$. Now the set $\calK \cap \frobball \subset \matR^{n \times n}$ is a $d$-dimensional unit ball w.r.t the $\norm{\cdot}_F$ norm, hence it follows\footnote{We simply treat a matrix $X \in \matR^{n \times n}$ by its vector form $\vect(X) \in \matR^{n^2}$.} from standard volumetric estimates (see for e.g. \cite[Corollary 4.2.13]{HDPbook}) that 
\begin{equation*}
    \calN(\calK \cap \frobball, \norm{\cdot}_F, \epsilon) \leq \left(\frac{3}{\epsilon} \right)^d \quad \text{ for } \epsilon  \leq 1
\end{equation*}
and $\calN(\calK \cap \frobball, \norm{\cdot}_F, \epsilon) = 1$ for $\epsilon > 1$. Since $\diam(\calK \cap \frobball) = 2$, it follows from \eqref{eq:gamma_bound} that for $\alpha = 1,2$, 
\begin{align*}
  \gamma_{\alpha}(\calK \cap \frobsphere, \norm{\cdot}_F) \lesssim \gamma_{\alpha}(\calK \cap \frobball, \norm{\cdot}_F) \lesssim \int_{0}^{1} d^{1/\alpha} \log^{1/\alpha}\left(\frac{3}{\epsilon} \right) d\epsilon \lesssim d^{1/\alpha},
\end{align*}
where the first inequality holds since $\calK \cap \frobsphere \subset \calK \cap \frobball$. This completes the proof.

%---------------------------------------------
% Proof of Corollary for l_1 ball case
%---------------------------------------------
\subsection{Proof of Corollary \ref{cor:sparse_example}} \label{subsec:proof_corr_sparse}
Since $A^*$ is $k$-sparse, we can bound $w(\tanconeAstar \cap \frobsphere)$ as 
\begin{equation*}
    w(\tanconeAstar \cap \frobsphere) \lesssim \sqrt{k \log(n^2/k) + k}  =: \beta(n,k);
\end{equation*}
see for instance \cite[Section 4.3]{Tropp2015} which can be directly applied to our setting. Then using \eqref{eq:talag_maj_meas_thm} and the fact $\norm{\cdot}_2 \leq \norm{\cdot}_F$, we obtain 
\begin{equation*}
    \gamma_2(\tanconeAstar \cap \frobsphere, \norm{\cdot}_2) \leq \gamma_2(\tanconeAstar \cap \frobsphere, \norm{\cdot}_F) \lesssim w(\tanconeAstar \cap \frobsphere) \lesssim \beta(n,k).
\end{equation*}
Now starting with the inequality $\gamma_1(\tanconeAstar \cap \frobsphere, \norm{\cdot}_2) \leq \gamma_1(\tanconeAstar \cap \frobsphere, \norm{\cdot}_F)$, we seek to bound the term $\gamma_1(\tanconeAstar  \cap \frobsphere, \norm{\cdot}_F)$ using \eqref{eq:gamma_bound}. To this end, we can use Sudakov's minoration inequality\footnote{see for instance \cite[Theorem 7.4.1]{HDPbook}; the canonical metric therein is $\norm{\cdot}_F$.} which yields for any $\epsilon > 0$,
\begin{equation} \label{eq:sparse_proof_bd1}
  \log \calN(\tanconeAstar \cap \frobsphere, \norm{\cdot}_F, \epsilon) \lesssim \frac{w^2(\tanconeAstar \cap \frobsphere)}{\epsilon^2} \lesssim \frac{\beta^2(n,k)}{\epsilon^2}.
\end{equation}
We can also obtain an alternate bound on $\calN(\tanconeAstar \cap \frobsphere, \norm{\cdot}_F, \epsilon)$ by using the following useful fact \cite[Exercise 4.2.10]{HDPbook}: for a metric space $(\calS,d)$ and $\calS' \subset \calS$, it holds for $\epsilon > 0$ that $\calN(\calS', d,\epsilon) \leq \calN(\calS, d,\epsilon/2)$. Applied to our setting, this yields the bound
\begin{align}  \label{eq:sparse_proof_bd2}
  \calN(\tanconeAstar \cap \frobsphere, \norm{\cdot}_F, \epsilon) \leq  \calN(\frobsphere, \norm{\cdot}_F, \epsilon/2) \leq \calN(\frobball, \norm{\cdot}_F, \epsilon/4) \leq \left(\frac{12}{\epsilon} \right)^{n^2},
\end{align}
where the final bound follows directly from \cite[Corollary 4.2.13]{HDPbook}. Hence from \eqref{eq:sparse_proof_bd1}, \eqref{eq:sparse_proof_bd2}, we have for all $\epsilon > 0$ that
\begin{equation*}
    \log \calN(\tanconeAstar \cap \frobsphere, \norm{\cdot}_F, \epsilon) \lesssim \min \set{\frac{\beta^2(n,k)}{\epsilon^2}, {n^2}\log \left(\frac{12}{\epsilon} \right)}.
\end{equation*}
Since $\diam(\tanconeAstar \cap \frobsphere) = 2$, we can employ \eqref{eq:gamma_bound} to obtain the bound 
\begin{align} \label{eq:sparse_proof_bd3}
    \gamma_1(\tanconeAstar \cap \frobsphere, \norm{\cdot}_F) \leq \int_{0}^{2} \log \calN(\tanconeAstar \cap \frobsphere, \norm{\cdot}_F, \epsilon) d\epsilon \lesssim 
    \underbrace{\int_{0}^{\epsilon_1} n^2 \log(12/\epsilon) d\epsilon}_{u_1} + \underbrace{\int_{\epsilon_1}^{2} \frac{\beta^2(n,k)}{\epsilon^2} d\epsilon}_{u_2},
\end{align}
for any $\epsilon_1 \in (0,2)$. It is easy to verify that 
\begin{align*}
u_1 &= n^2 \epsilon_1 \left(\log\left(\frac{12}{\epsilon_1}\right) + 1 \right) \leq 2 n^2 \epsilon_1 \log\left(\frac{12}{\epsilon_1}\right), \\
\text{ and } u_2 &= \left(\frac{2-\epsilon_1}{2\epsilon_1} \right) \beta^2(n,k).
\end{align*}
Using this in \eqref{eq:sparse_proof_bd3}, and choosing $\epsilon_1 = \beta(n,k)/n$, we obtain the bound 
\begin{align*}
    \gamma_1(\tanconeAstar \cap \frobsphere, \norm{\cdot}_F) \lesssim n \beta(n,k) \left[\log\left(\frac{n}{\beta(n,k)} \right) + 1 \right].
\end{align*}
The statement of the corollary now follows directly. Note that for the above choice of $\epsilon_1$ to be in the interval $(0,2)$, we need $n \geq \beta(n,k)/2$ to hold but this is always satisfied.

%-----------------------
% Convex regression
%-----------------------
\subsection{Proof of Corollary \ref{cor:convex_reg_biv}} \label{subsec:proof_cor_convexreg_biv}
The proof relies on using metric entropy bounds developed in \cite[Section 4]{Kur2020ConvexRI} for the collection of convex functions lying within a ball around $f_0$. In what follows, we will borrow some notations and definitions from \cite{Kur2020ConvexRI}.

Let $\calS$ denote an equispaced $n \times n$ grid of $\Omega = [0,1/\sqrt{2}]^2$ with spacing $\delta = \frac{1}{(n-1)\sqrt{2}}$ along each axis. Note that $\Omega$ lies within the unit $\ell_2$ ball of $\matR^2$ (centered at origin). This is a requirement for the results in \cite{Kur2020ConvexRI} to hold, but can be ensured by a suitable scaling. For any function defined on $\Omega$, let us define the quasi-norm
\begin{equation*}
\ell_{\calS}(f,\Omega,p) = \left(\frac{1}{\abs{\Omega \cap \calS}} \sum_{b \in \Omega \cap \calS} \abs{f(b)}^p \right)^{1/p}.
\end{equation*}
Then for any convex $f_0$ defined on $\Omega$, and $x > 0$, define a ball  centered at $f_0$,
\begin{equation*}
    \calB_{\calS}^p(f_0;x;\Omega) = \set{f: \Omega \rightarrow \matR \text{ s.t $f$ is convex and } \ell_{\calS}(f-f_0,\Omega,p) \leq x}.
\end{equation*}
Now denoting $\calKtil = \calK - A^*$, note that
\begin{equation*}
    \calKtil \cap x \frobball = \set{A - A^*: \ A \in \calK, \ \norm{A - A^*}_F \leq x}.
\end{equation*}
We need to bound the metric entropy $\log \calN(\calKtil \cap x \frobball, \norm{\cdot}_F, \epsilon)$. To this end, it is not difficult to verify that 
\begin{equation} \label{eq:covreg_tmp1}
    \calN(\calKtil \cap x \frobball, \norm{\cdot}_F, \epsilon) = \calN\left(\calB_{\calS}^2\left(f_0;\frac{x}{n};\Omega \right), \ell_{\calS}(\cdot,\Omega,2), \frac{\epsilon}{n} \right).
\end{equation}
The RHS of \eqref{eq:covreg_tmp1} is bounded by \cite[Corollary 4.2]{Kur2020ConvexRI}, which in our setup implies
\begin{equation*}
    \log \calN(\calKtil \cap x \frobball, \norm{\cdot}_F, \epsilon) = \log \calN\left(\calB_{\calS}^2\left(f_0;\frac{x}{n};\Omega \right), \ell_{\calS}(\cdot,\Omega,2), \frac{\epsilon}{n} \right) \leq c^s K \left(\frac{x}{\epsilon}\right) \log^s n
\end{equation*}
for some constant $c > 0$. Moreover, using standard volumetric estimates as earlier (e.g., \cite[Proposition 4.2.12, Corollary 4.2.13]{HDPbook}), we also have
$\log \calN(\calKtil \cap x \frobball, \norm{\cdot}_F, \epsilon) \leq n^2 \log(6x/\epsilon)$, thus implying 
\begin{equation*}
    \log \calN(\calKtil \cap x \frobball, \norm{\cdot}_F, \epsilon) \leq \min \set{c^s K \left(\frac{x}{\epsilon}\right) \log^s n, n^2 \log\left(\frac{6x}{\epsilon} \right)}.
\end{equation*}
Now it remains to bound the gamma functionals by deploying \eqref{eq:gamma_bound}.
\begin{enumerate}
\item ({\bf Bounding $\gamma_1(\calKtil \cap x\frobball, \norm{\cdot}_2)$}). Since $\diam(\calKtil \cap x\frobball) = 2x$, we have using \eqref{eq:gamma_bound}, for any $0 < \epsilon_1 < 2x$, that
\begin{align*}
    \gamma_1(\calKtil \cap x\frobball, \norm{\cdot}_2) \leq \gamma_1(\calKtil \cap x\frobball, \norm{\cdot}_F) 
    &\lesssim \int_0^{2x} \log \calN(\calKtil \cap x \frobball, \norm{\cdot}_F, \epsilon) d\epsilon \\
    &\lesssim n^2 \int_0^{\epsilon_1} \log\left(\frac{6x}{\epsilon}\right) d\epsilon + c^s K x \log^s n \int_{\epsilon_1}^{2x} \frac{1}{\epsilon} d\epsilon \\
    &\lesssim n^2 \epsilon_1 \log\left(\frac{6x}{\epsilon_1}\right) + c^s Kx \log^s n \log\left(\frac{2x}{\epsilon_1}\right).
\end{align*}
In the final inequality above, we used similar calculations as within the proof of Corollary \ref{cor:sparse_example}. Choosing $\epsilon_1 = \frac{c^s K x \log^s n}{n^2}$ then implies
\begin{align*}
 \gamma_1(\calKtil \cap x\frobball, \norm{\cdot}_F)  \lesssim  x c^s K \log^s n \log\left(\frac{n^2}{c^s K \log^s n} \right) 
\end{align*}
where we note that $\epsilon_1 < 2x$ is ensured provided $n^2 \geq K c^s \log^s n$.

\item ({\bf Bounding $\gamma_2(\calKtil \cap x\frobball, \norm{\cdot}_2)$}). Again using \eqref{eq:gamma_bound}, we obtain
\begin{align*}
    \gamma_2(\calKtil \cap x\frobball, \norm{\cdot}_2) 
    &\leq \gamma_2(\calKtil \cap x\frobball, \norm{\cdot}_F) \\
    &\lesssim \int_0^{2x} \log^{1/2} \calN(\calKtil \cap x \frobball, \norm{\cdot}_F, \epsilon) d\epsilon \\
    &\lesssim \int_0^{2x} c^{s/2} \frac{(Kx)^{1/2}}{\sqrt{\epsilon}} \log^{s/2} n \ d\epsilon \\
    &\lesssim x c^{s/2} K^{1/2} \log^{s/2} n.
\end{align*}
\end{enumerate}
Plugging these bounds in Theorem \ref{thm:no_tancone_Struc}, we see that the dependence on $x$ cancels out from all terms barring the additive term at the end. Since Theorem $2$ is true for any $x > 0$, we can now take the limit $x = 0$ to obtain the corollary.

%-----------------------
% Lipschitz regression
%-----------------------
\subsection{Proof of Corollary \ref{cor:lipschitz_reg_examp}} \label{subsec:proof_cor_lipreg}

%\rev{
We start by constructing $B \in \calK$ that is close to $A^*$, and each row of which is ``piecewise affine''. We will denote $X_{i,:}$ to be the $i$th row of a matrix $X$. 
\begin{lemma} \label{lem:approx_affine}
    For any $A \in \calK$, there exists $B \in \calK$ such that $B_{i,:}$ has at most $l_i + 1$ affine pieces, slopes equal to $\pm L_i$, and $\abs{B_{i,j} - B_{i,j+1}} = L_i/T$ for all $i,j$ so that 
    \begin{equation*}
        \norm{A - B}_F^2 \leq \frac{4n^3}{T^2} \left(\sum_{i=1}^n \frac{L_i^2}{l_i^2} \right).
    \end{equation*}
\end{lemma}
\begin{proof}
    The proof follows by using the arguments in \cite[Lemma 3.1]{neykov19b} to approximate $A_{i,:}$ by $B_{i,:}$ where $B_{i,:}$ has at most $l_i + 1$ pieces and the length of each piece is at most $k_i := \lfloor n/l_i \rfloor$, such that 
    \begin{equation*}
        \abs{B_{i,j} - A_{i,j}} \leq 2k_i \frac{L_i}{T} \leq \frac{2nL_i}{l_i T}.
    \end{equation*}
    This implies the stated bound in the lemma, the details are omitted.
\end{proof}
We will next bound the Gaussian width $w(\tanconeB \cap \frobball)$ by following the steps in \cite[Section 3.1]{neykov19b}. To this end, denote the monotone sequence cone $\monot_m := \set{v \in \matR^m: v_1 \leq v_2 \leq \cdots \leq v_m}$. The following result follows directly from \cite[Lemma 3.2]{neykov19b}.
\begin{lemma} \label{lem:tancone_subset}
    Let $B \in \calK$ be as in Lemma \ref{lem:approx_affine}. For each $i \in [n]$, let $T_{i,1}, \dots, T_{i,l_i + 1}$ be a disjoint partition of $[n]$ such that $B_{i,:}$ is piecewise affine on each $T_{i,j}$ with $s_{i,j}$ denoting the sign of the slope (for $j=1,\dots,l_i+1$). Then 
    \begin{equation*}
        \tanconeB \subset \set{U \in \matR^{n \times n}: U_{i,:} \in \left(- s_{i,1} \monot_{\abs{T_{i,1}}} \right) \times \cdots \times \left(- s_{i,l_i + 1} \monot_{\abs{T_{i,l_i+1}}} \right)} =: \mathcal{C}_{B}.
    \end{equation*}
\end{lemma}
Now recall the definition of the statistical dimension of a \revjf{convex} cone \cite{lotz14} $\mathcal{C} \subset \matR^n$, denoted by
\begin{equation*}
    \nu(\mathcal{C}) :=  \expec\left[\left(\sup_{v \in \mathcal{C}, \norm{v}_2 \leq 1} v^\top g \right)^2 \right]
\end{equation*}
for $g \sim \mathcal{N}(0,I_n)$. \revjf{ Note that $\tanconeB, \mathcal{C}_{B}$ are closed convex cones. For closed convex cones $\calC_1, \calC_2 \subset \matR^m$ it was shown in \cite{lotz14} that $\calC_1 \subset \calC_2$ implies $\nu(\calC_1) \le \nu(\calC_2)$. It was also shown that $\nu(\calC_1 \times \calC_2) = \nu(\calC_1) + \nu(\calC_2).$} Since \revjf{$\vect(\tanconeB^\top), \vect(\mathcal{C}_{B}^\top)$ are closed convex cones and} $\vect(\tanconeB^\top) \subset \vect(\mathcal{C}_{B}^\top)$, this implies $\nu(\vect(\tanconeB^\top)) \leq \nu(\vect(\mathcal{C}_{B}^\top))$; here $\vect(\cdot)$ and transpose operations are applied element-wise. It then follows that 
\begin{align*}
    \revj{\nu(\vect(\tanconeB^\top))} \leq \revj{\nu(\vect(\mathcal{C}_{B}^\top))} = \sum_{i=1}^{n} \sum_{j=1}^{l_i+1} \nu(-s_{i,j} \monot_{\abs{T_{i,j}}}) = \sum_{i=1}^{n} \sum_{j=1}^{l_i+1} \nu(\monot_{\abs{T_{i,j}}})
    \leq \sum_{i=1}^{n} (l_i + 1)  \log\left(\frac{e n}{l_i + 1} \right)
\end{align*}
where the last inequality follows from bounds on the statistical dimension of monotone cones \cite{lotz14} and Jensen's inequality. By Cauchy-Schwarz, we then have that 
\begin{equation*}
   w(\tanconeB \cap \frobball) \leq \sqrt{\expec \left[\left(\sup_{V \in \tanconeB \cap \frobball} \dotprod{G}{V} \right)^2 \right]} = \sqrt{\nu(\vect(\tanconeB^\top))} \leq \left(\sum_{i=1}^{n} (l_i + 1)  \log\left(\frac{e n}{l_i + 1} \right) \right)^{1/2}. 
\end{equation*}
It then follows using \eqref{eq:talag_maj_meas_thm} that 
\begin{align}\label{eq:gamma_2_bds_lip}
    \gamma_2(\tanconeB \cap \frobsphere, \norm{\cdot}_F) \lesssim \gamma_2(\tanconeB \cap \frobball, \norm{\cdot}_F) \lesssim \left(\underbrace{\sum_{i=1}^{n} (l_i + 1)  \log\left(\frac{e n}{l_i + 1} \right)}_{=: \alpha_n} \right)^{1/2}.
\end{align}
It remains to bound $\gamma_1(\tanconeB \cap \frobsphere, \norm{\cdot}_F)$. As in the proofs of the previous corollaries, we use Sudakov minoration along with covering number bounds for $\frobball$ to obtain (for any $\epsilon > 0$)
\begin{equation*}
    \log \calN(\tanconeB \cap \frobball, \norm{\cdot}_F, \epsilon) \lesssim \min \set{\frac{\alpha_n}{\epsilon^2}, n^2 \log \left(\frac{12}{\epsilon} \right)}.
\end{equation*}
Then we have for any $\epsilon \in (0,2)$
\begin{align*}
    \gamma_1(\tanconeB \cap \frobball, \norm{\cdot}_F) \lesssim \int_{0}^{\epsilon_1} n^2 \log(12/\epsilon) d\epsilon + \int_{\epsilon_1}^{2} \frac{\alpha_n}{\epsilon^2} d\epsilon.
\end{align*}
Choosing $\epsilon_1 = \frac{\sqrt{\alpha_n}}{n}$, we obtain (analogous to the calculation in Section \ref{subsec:proof_corr_sparse})
\begin{align}\label{eq:gamma_1_bds_lip}
    \gamma_1(\tanconeB \cap \frobball, \norm{\cdot}_F) \lesssim n\sqrt{\alpha_n} \left[\log(\frac{n}{\sqrt{\alpha_n}}) + 1 \right].
\end{align}
For simplicity, we set $l_i = l \ \revj{\revjf{(\ \geq 1)}}$ and use $L_i \leq L$ for each $i \in [n]$. Using the estimates in \eqref{eq:gamma_2_bds_lip}, \eqref{eq:gamma_1_bds_lip} in Theorem \ref{thm:main_err_tangent_cone}, we have for 
\begin{equation} \label{eq:T_Lipreg_bd}
    T \gtrsim J^4(A^*) \max \set{n(l+1) \log(\frac{en}{l+1}), \log^2(1/\delta)}
\end{equation}
that with probability at least $1-\delta$,
\begin{align*}
    \norm{\est{A} - A^*}_F 
    &\lesssim J(A^*) \left(\frac{\log(1/\delta) + \sqrt{n(l+1) \log(\frac{en}{l+1})}}{\sqrt{T}} + \frac{n^{3/2} (l+1)^{1/2} \log^{3/2}(\frac{en}{l+1})}{T}\right) + J^{2}(A^*) \revj{\frac{n^2 L}{T l}} \\
    &\lesssim J^2(A^*) \left[ \frac{\log(1/\delta)}{\sqrt{T}} + \sqrt{l} \left(\max\set{\sqrt{\frac{n\log n}{T}}, \frac{(n \log n)^{3/2}}{T}} \right) +   \revj{\frac{n^2 L}{T l}} \right].
\end{align*}
Now for the choice 
\begin{equation} \label{eq:l_bd_Lipreg}
    l = \left \lceil \left(\frac{\revj{n^2L/ T}}{\max\set{\sqrt{\frac{n \log n}{T}}, \frac{(n\log n)^{3/2}}{T}}} \right)^{\revj{2/3}} \right\rceil 
    = \left \lceil \frac{\revj{\frac{n L^{2/3}}{T^{1/3} (\log n)^{1/3}}}}{\left(\max\set{1,\frac{n\log n}{\sqrt{T}}} \right)^{\revj{2/3}}} \right\rceil = \revjf{\lceil \beta \rceil}
\end{equation}
we \revj{readily} obtain the stated bound on $\norm{\est{A} - A^*}_F$ in the corollary \revjf{ after some simplifications using the facts $l \geq \beta$ and $l \leq \beta + 1$.} 

Using \eqref{eq:l_bd_Lipreg} in \eqref{eq:T_Lipreg_bd}, we obtain after some simplifications the sufficient condition in \eqref{eq:T_bd_lipreg_fin}.
%}
\revj{To see this, note that \eqref{eq:T_Lipreg_bd} holds provided $T \gtrsim J^4(A^*) \max \set{nl \log n, \log^2(1/\delta)}$. Using \eqref{eq:l_bd_Lipreg} \revjf{ along with $l \leq \beta + 1$}, and the fact $\max\set{a,b} \asymp a + b$ for $a,b \geq 0$, this means that it suffices that,}
\revj{
\begin{align*}
    T \gtrsim J^4(A^*) \max \set{\revjf{n \log n + }\frac{n^2 L^{2/3} (\log n)^{2/3}}{T^{1/3} + (n \log n)^{2/3}}, \log^2(1/\delta)}.
\end{align*}}
\revj{The above condition is readily verified to hold provided \eqref{eq:T_bd_lipreg_fin} holds.}

% Acknowledgment
\section*{Acknowledgment}
H.T would like to thank Yassir Jedra and Sjoerd Dirksen for helpful clarifications regarding their results in their works \cite{Jedra20}, \cite{dirksen15} respectively. \revo{The first version of this paper was prepared when H.T was affiliated to Inria Lille. \revj{H.T was supported by a Nanyang Associate Professorship (NAP) grant from NTU Singapore.}}

\newpage
\bibliographystyle{plain}
\bibliography{references}

\end{document}